\documentclass{amsart}
\usepackage[T1]{fontenc}
\usepackage{lmodern}
\usepackage{graphicx}
\usepackage[english]{babel}
\usepackage{mathrsfs}
\usepackage[bbgreekl]{mathbbol}
\usepackage{amsmath, amsfonts, amssymb}
\usepackage{stmaryrd}
\usepackage{mathrsfs}
\usepackage{mathtools}
\usepackage{listings}
\usepackage{floatrow}
\usepackage{fourier-orns}
\usepackage{faktor}
\usepackage[all]{xy}
\usepackage[colorlinks = True , allcolors = blue]{hyperref}

\theoremstyle{plain}
\newtheorem{theorem}{Theorem}[section]
\newtheorem{lemma}[theorem]{Lemma}
\newtheorem{corollary}[theorem]{Corollary}
\newtheorem{proposition}[theorem]{Proposition}
\theoremstyle{definition}
\newtheorem{definition}[theorem]{Definition}

\theoremstyle{remark}
\newtheorem{remmark}[theorem]{Remark}
\newtheorem{example}[theorem]{Example}

\numberwithin{equation}{section}

\newcommand{\tsp}[1]{{}^t \! #1}

\newcommand{\ii}{\mathrm{i}}
\newcommand{\tr}{\operatorname{Tr}}
\newcommand{\supp}{\operatorname{supp}}
\newcommand{\im}{\operatorname{im}}
\newcommand*\diff{\mathop{}\!\mathrm{d}}

\newcommand{\Id}{\operatorname{Id}}
\newcommand{\R}{\mathbb R}
\newcommand{\N}{\mathbb N}

\newcommand{\Z}{\mathbb Z}
\newcommand{\CC}{\mathbb C}

\newcommand{\CCC}{\mathscr{C}}
\newcommand{\T}{\mathbb T}
\newcommand{\PP}{\mathbb{P}}
\newcommand{\F}{\mathcal F}

\newcommand{\Aut}{\operatorname{Aut}}
\newcommand{\Ad}{\operatorname{Ad}}

\newcommand{\G}{\mathcal{G}}
\newcommand{\g}{\mathfrak{g}}
\newcommand{\f}{\mathfrak{f}}
\newcommand{\gt}{\mathfrak{t}}
\newcommand{\vphi}{\varphi}
\newcommand{\rr}{\rightrightarrows}
\newcommand{\act}{\curvearrowright}
\newcommand{\lag}{\langle}
\newcommand{\rag}{\rangle}
\newcommand{\op}{\operatorname{Op}}
\newcommand{\ttt}{\mathfrak{t}}
\newcommand{\Lie}{\operatorname{Lie}}
\newcommand{\hd}{\Omega^{1/2}}
\newcommand{\rk}{\operatorname{rk}}
\newcommand{\m}{\operatorname{\mathcal{M}}}
\newcommand{\pr}{\operatorname{pr}}
\newcommand{\ev}{\operatorname{ev}}
\newcommand\isomto{\stackrel{\sim}{\smash{\longrightarrow}\rule{0pt}{0.4ex}}}

\begin{document}

\title[Filtered calculus and crossed product]{Filtered Calculus and crossed products by \texorpdfstring{$\R$}{R}-actions}

\author[Cl\'ement Cren]{Cl\'ement Cren}

\address{Cl\'ement Cren, Univ Paris Est Creteil, Univ Gustave Eiffel, CNRS, LAMA UMR8050, F-94010 Creteil, France}
\email{\href{mailto:clement.cren@u-pec.fr}{clement.cren@u-pec.fr}}
%\begin{subjclass}
%\end{subjclass}

%\begin{keywords}
%\end{keywords}

%\blfootnote{\textit{MSC Classification: }Primary: 47G30, 58H05 ; Secondary: 46L89, 58J22, 22E25}
\begin{abstract}
We show an isomorphism between the kernel of the $C^*$ algebra of the tangent groupoid of a filtered manifold and the crossed product of the order $0$ pseudodifferential operators in the associated filtered calculus by a natural $\R$-action. This isomorphism is constructed in the same way as in the classical pseudodifferential calculus by Debord and Skandalis. The proof however relies on a structure result for the $C^*$-algebra of graded nilpotent Lie groups which did not appear in the commutative case and extends a result of Epstein and Melrose in the case of contact manifolds.
\end{abstract}
\maketitle

\section*{Introduction}

This paper deals with an alternate viewpoint of calculus on filtered manifolds. Filtered manifolds are manifolds equipped with a filtration of their tangent bundle by subbundles with a compatibility with the Lie bracket of vector fields. These manifolds have an alternate pseudodifferential calculus, which allows the construction of parametrices for hypoelliptic, but not elliptic, differential operators. Classical examples gather contact \cite{EpsteinMelrose,BaumVanerp}, CR \cite{FollandSteinEstimate}, and Heisenberg manifolds \cite{bealsgreiner}, but also foliated manifolds \cite{HilsumSkandalis,ConnesMoscovici}. More recently, manifolds with parabolic geometries were also considered \cite{dave2017graded}. These new advances are all thanks to a new approach on the pseudodifferential calculus using groupoid methods \cite{van_erpyunken, AndroulidakisMohsenYunken} (there is another approach relying on harmonic analysis techniques in \cite{FermanianFischer}).
We adapt a construction of Debord and Skandalis \cite{debord2014adiab} to the filtered case. They showed for the classical calculus an isomorphism between the $C^*$-algebra of pseudodifferential operators (of order 0) on a manifold $M$ and an algebra constructed from the tangent groupoid and its zoom action. Indeed if $\T M$ is the tangent groupoid of $M$ we can restrict it to $\T^+M \rr M \times \R_+$. It then sits in the exact sequence:
$$\xymatrix{0 \ar[r] & \CCC_0(\R^*_+,\mathcal{K}) \ar[r] & C^*(\T^+M) \ar[r] & \CCC_0(T^*M) \ar[r] & 0.}$$
Now the zoom action $(\alpha_{\lambda})_{\lambda > 0}$ is defined on $\T^+M$ via:
\begin{align*}
\alpha_{\lambda}(x,y,t) &= (x,y,\lambda^{-1}t) &(x,y,t) \in M \times M \times \R^*_+\\
\alpha_{\lambda}(x,\xi,0) &= (x,\lambda\xi,0) &(x,\xi) \in TM.
\end{align*}
If we remove the zero section from the fiber at zero we get the exact sequence:
$$\xymatrix{0 \ar[r] & \CCC_0(\R^*_+,\mathcal{K}) \ar[r] & C^*_0(\T^+M) \ar[r] & \CCC_0(T^*M\setminus M) \ar[r] & 0.}$$
Now the zoom action on \(T^*M\setminus M\) becomes free and proper. Taking the cross product yields the exact sequence:
$$\xymatrix{0 \ar[r] & \mathcal{K} \otimes \mathcal{K} \ar[r] & C^*_0(\T^+M)\rtimes_{\alpha} \R^*_+ \ar[r] & \CCC(\mathbb{S}^*M)\otimes \mathcal{K} \ar[r] & 0.}$$
Debord and Skandalis related this sequence to the one of pseudodifferential operators:
$$\xymatrix{0 \ar[r] & \mathcal{K} \ar[r] & \Psi^*(M) \ar[r] & \CCC(\mathbb{S}^*M) \ar[r] & 0}$$
where $\Psi^*(M)$ is the closure of order $0$ operators in the algebra of bounded operators on $L^2(M)$ and the arrow on the right is the principal symbol map. As an important step, they showed that any pseudodifferential operator of order $m$ could be written as an integral of the form:
$$\int_0^{+\infty}t^{-m}f_t\frac{\diff t}{t}$$
where $(f_t)_{t\geq 0}$ is a function in a suitable ideal of the Schwartz algebra of the tangent groupoid.
In the context of filtered manifold one could perform the same type of construction. We thus define the appropriate Schwartz algebra on the tangent groupoid and show that pseudodifferential operators can be written with an integral \textit{à la} Debord-Skandalis. This time however, the identification at the level of symbols is not trivial at all. To do this we use Kirillov's theory and Pedersen's stratification of the coadjoint orbits to decompose the algebra $\Sigma(G)$ of principal symbols of order $0$ on a graded Lie group $G$ into a finite sequence of nested ideals:
$$0 = \Sigma_0(G) \triangleleft \Sigma_1(G) \triangleleft \cdots \triangleleft \Sigma_r(G) = \Sigma(G),$$
such that for all $i \geq 0$ we have $\faktor{\Sigma_i(G)}{\Sigma_{i-1}(G)} \cong \CCC_0(\faktor{\Lambda_i}{\R^*_+},\mathcal{K}_i)$. Here $\Lambda_i$ is a certain subset of the set of coadjoint orbits $\faktor{\g^*}{G}$, so that the set $\faktor{\Lambda_i}{\R^*_+}$ is locally compact Hausdorff, and $\mathcal{K}_i$ is an algebra of compact operators. We also extend Pedersen's stratification to smooth families of graded groups in Section \ref{Section:Stratification for Family}, thus obtaining a decomposition of the algebra of principal symbols for families.  This decomposition generalizes the decomposition of symbols of Epstein and Melrose in the case of contact manifolds (i.e. for the Heisenberg group) \cite{EpsteinMelrose}. Combining this with the construction on the tangent groupoid gives an isomorphism:
$$\Psi_H^*(M) \rtimes \R \cong C^*_0(\T^+_HM).$$
Here $\T_H^+M$ is the restriction over $M\times \R_+$ of the tangent groupoid of the filtered manifold $(M,H)$ and the $\R$ action on the $C^*$-algebra of order $0$ pseudodifferential operators in the filtered calculus is constructed so that the isomorphism is $\R^*_+$-equivariant (for the dual action on the left and the zoom action on the right). As a corollary of our construction, we derive a KK-equivalence between the algebras of pseudodifferential operators of order $0$ for the classical and filtered calculus, extending the one for the symbols obtained in \cite{mohsen2020index} from ideas of \cite{BaumVanerp}. This equivalence shows that one cannot hope to obtain new K-theoretic invariants from the filtered calculus (other than the ones that were previously obtainable with the classical calculus). The majority of the results in this article were also obtained by Ewert in \cite{Ewert} using generalized fixed point algebras methods. Her version of the calculus also differs from the one used here (although the $C^*$-completions of non-positive order operators are the same). We directly show the equivalence of the van Erp-Yuncken approach and the Debord-Skandalis approach of pseudodifferential calculus. We also use a different decomposition of the symbol algebra than the one in \cite{Ewert} giving a nicer description of the subquotients in the symbol algebra. Finally we show an isomorphism $\Psi^*_H(M)\rtimes \R \cong C^*_0(\T^+_HM)$ which is stronger than the Morita equivalence $C^*_0(\T^+_HM)\rtimes \R^*_+ \sim \Psi^*_H(M)$ obtained in \cite{Ewert}.

\section{Representation theory of graded nilpotent groups}

Let $G$ be a graded nilpotent Lie group. This means that $G$ is a simply connected Lie group and that its lie algebra $\g$ admits a decomposition $\g = \bigoplus_{j = 1}^r \g_j$ compatible with the Lie bracket in the following sense: $\forall j,k, \left[ \g_j,\g_k \right] \subset \g_{j+k}$ (we make the convention that $\g_{\ell} = \{0\}$ whenever $\ell > r$). Given such a group, one can define a family of inhomogeneous dilations $\delta_{\lambda} \in \Aut(G)$ for $\lambda \in \R^*_+$. At the level of the Lie algebra $\g$ they are linear maps given by:
$$\forall \lambda >0, \forall 1\leq k \leq r, \forall X \in \g_k, \diff \delta_{\lambda}(X) = \lambda^k X.$$
The map $\diff\delta_{\lambda}$ is well defined and is a Lie algebra automorphism and hence lifts to a Lie group automorphism of $G$ denoted $\delta_{\lambda}$. 

\begin{example}Let $G = \R^n$ with the trivial grading (everything has degree 1), we get the usual dilations $\diff\delta_{\lambda}(X) = \lambda X$ for every $X\in \R^n$.\end{example}

\begin{example}Let $G = H_{2n+1}$ be the Heisenberg group of dimension $2n+1$. Its Lie algebra is generated by elements $X_1,\cdots,X_n,Y_1,\cdots,Y_n,Z$ with the commutation relations $[X_j,Y_k] = \delta_{j,k}Z$ and all the other brackets being zero. The decomposition of its Lie algebra $\g$ is given by $\g_1 = \left\langle X_1,\cdots,X_n,Y_1,\cdots,Y_n \right\rangle$ and $\g_2 = \left\langle Z \right\rangle$. The dilation is thus given by $\diff\delta_{\lambda}(X_k) = \lambda X_k, \diff\delta_{\lambda}(Y_k) = \lambda Y_k, \diff\delta_{\lambda}(Z) = \lambda^2 Z$.
\end{example}

\begin{remmark}
    In the literature, graded Lie groups are also called Carnot groups. The terminology of Carnot groups might also refer specifically to the graded Lie groups for which the Lie algebra is generated by \(\g_1\). This specific type of Lie groups and algebras appears in sub-Riemannian geometry (see e.g. \cite{GromovCarnot}).
\end{remmark}

The action of $\R^*_+$ on $G$ gives by duality an action on $\widehat{G}$. This action, still denoted by $\delta_{\lambda}$, stabilizes the trivial representation $1$. The restricted action on $\widehat{G} \setminus \{1\}$ is free and proper. We wish to understand the structure of $C^*_0(G)\rtimes \R^*_+$. One could hope to write $C^*_0(G)$ as some crossed product over $\R$, $A \rtimes \R$, for a $\R-C^*$-algebra $A$ so that $\delta_{\lambda}$ would become the dual action.

In the commutative case $C^*_0(\R^n) = \CCC_0(\R^n\setminus\{0\})$. We thus have $C^*_0(\R^n)\rtimes \R^*_+ = \CCC\left( \faktor{\R^n\setminus \{0\}}{\R^*_+}\right) \otimes \mathcal{K}(L^2(\R)) = \CCC(\mathbb{S}^{n-1})\otimes \mathcal{K}(L^2(\R))$. In \cite{debord2014adiab}, Debord and Skandalis showed that this algebra naturally appears as the algebra of principal symbols of pseudodifferential operators (of order $0$). We wish to do the same here in the general setting of graded nilpotent groups. 

To better understand the structure of $C^*_0(G)$ we study the representation theory of $G$ through Kirillov's theory \cite{Kirillov} and the stratification of its unitary dual $\widehat{G}$ constructed by Pukánszky \cite{Puck}.

Kirillov's orbit method gives a homeomorphism between $\widehat{G}$ equipped with the hull-kernel topology and the set of coadjoint orbits $\faktor{\g^*}{\Ad^*(G)}$ equipped with the quotient of the euclidean topology. This homeomorphism is $\R^*_+$-equivariant for the natural actions obtained by composition with $\delta_{\lambda}$ on the representations and by the transpose of $\diff \delta_{\lambda}$ on $\g^*$. 
Consider $(X_1,\cdots,X_n)$, a basis of $\g$ consisting of eigenvectors for the dilations, i.e. $\diff\delta_{\lambda}(X_i) = \lambda^{q_i} X_i$. Here $q_i$ is the integer such that $X_i \in \g_{q_i}$. If the $q_i$ are in ascending order then $(X_1,\cdots,X_n)$ is a Jordan-Hölder basis. Following Pukánszky \cite{Puck} and Pedersen \cite{PedersenQuant} (see also \cite{CorwinGreenleaf}), we obtain a decomposition of $\g^* \setminus \{0\}$ with the following properties:

\begin{theorem}[Pukánszky, Pedersen] There exist an integer $d$, and algebraic sets $\Omega_i \subset \g^*\setminus \{0\}, 1\leq i \leq d $, that are both $G$ and $\R^*_+$-invariant, such that:
\begin{itemize}
\item the sets $W_i = \bigcup_{1\leq j \leq i}\Omega_{j}$ are Zariski open, invariant for the actions of $G$ and $\R^*_+$, and $W_d=\g^* \setminus \{0\}$.
\item for each $i$ there exist complementary vector subspaces of $\g^*$: $V_i^S$ and $V_i^T$, such that each $G$-orbit in $\Omega_i$ meets $V_i^T$ in exactly one point. The sets $C_i = \Omega_i \cap V_i^T$ are $\R^*_+$-invariant algebraic subsets and define $\R^*_+$-equivariant cross-sections for the $G$-orbits in $\Omega_i$.
\item there is a birational nonsingular bijection $\psi_i \colon C_i \times V_i^S \to \Omega_i$ such that if $p_i^T$ denotes the projection onto $V_i^T$ parallel to \(V_i^S\), then for each $v \in \Sigma_i$, the map $p_i^T \circ \psi_i(v,\cdot) \colon V_i^S \to V_i^T$ is a polynomial, whose graph is the orbit $G\cdot v$. The composition $\pr_1 \circ \psi_i^{-1} \colon \Omega_i \to C_i$ is $G$-invariant and its quotient through $G$ is the aforementioned cross-section.
\end{itemize}
\end{theorem}

\begin{proof}
The proof can be found in \cite{Puck} for an arbitrary nilpotent Lie group. Without reproving the whole theory, we will give in Section \ref{Section:Stratification for Family} the construction of the sets \(\Omega_j\) in order to expand this stratification to any smooth family of graded groups. In our context, our choice of basis makes all the elements in the construction stable under the $\R^*_+$-action. For instance, the jump indices of an element in $\g^*$ are the same as the other elements of its orbit through the $\R^*_+$-action. Therefore the sets $\Omega_i$, $W_i$, $C_i$... are stable under the $\R^*_+$-action and the maps $\psi_i$ are equivariant. 
\end{proof}

\begin{corollary}With the same notations, for every $1 \leq i \leq d$ the natural map $C_i \to \faktor{\Omega_i}{G}$ is a $\R^*_+$-equivariant homeomorphism. In particular the action $\R^*_+ \act \faktor{\Omega_i}{G}$ is free and proper.\end{corollary}

\begin{proof}
The map is a cross section of the action, therefore a homeomorphism. Its equivariance was also stated in the previous theorem. The action $\R^*_+\act C_i$ is the restriction of an action on the vector space $V_i^T \setminus \{0\}$ (remember that $0$ corresponds to the trivial representation that we have removed). This action is made of diagonal matrices of dilations by powers of $\lambda > 0$. It is therefore free and proper. The equivariant homeomorphism then transfers this result to $\R^*_+ \act \faktor{\Omega_i}{G}$.
\end{proof}

\begin{corollary}\label{Stratification}There exists a filtration $\emptyset = V_0 \subset V_1 \subset \cdots \subset V_d = \widehat{G}\setminus\{0\}$ by open subsets such that the each $\Lambda_i = V_i \setminus V_{i-1}$ is Hausdorff, $\R^*_+$-invariant and the action $\R^*_+ \act \Lambda_i$ is free and proper. In particular the sets $\faktor{\Lambda_i}{\R^*_+}$ are Hausdorff.
\end{corollary} 

\begin{proof}
Define $V_i = \faktor{W_i}{G}$, it defines a filtration of $\widehat{G}\setminus \{0\}$ through Kirillov's map. Moreover the previous corollary gives a homeomorphism between $\Lambda_i$ and $C_i$ which is thus Hausdorff. Since the action $\R^*_+\act C_i$ is free and proper and the homeomorphism $C_i \cong \Lambda_i$ is $\R^*_+$-equivariant, then the action $\R^*_+\act \Lambda_i$ is also free and proper.
\end{proof}

\begin{remmark}In both cases, the last stratum corresponds to the elements in $\g^*$ whose corresponding representation is one dimensional. All the other unitary irreducible representations are infinite dimensional because of Engel's theorem.\end{remmark}

This description of Pedersen's fine stratification, without the details on the construction, does not show the differences with Pukánszky's coarse stratification. Without going into too much detail, we will use the following properties of Pedersen's stratification:

\begin{theorem}[Pedersen \cite{PedersenQuant} ; Lipsman, Rosenberg \cite{LipRos}]Let $\Omega_i\subset \g^*$ be a stratum of Pedersen's fine stratification. Then there exists an integer $n \in \N$ such that every representation in the corresponding set $\Lambda_i = \faktor{\Omega_i}{G} \subset \widehat{G}$ can be realized on $L^2(\R^n)$ with smooth vectors being the Schwartz functions $\mathcal{S}(\R^n)$ and the elements of $\g$ are realized as differential operators with polynomial coefficients. These coefficients have a (non-singular) rational dependence when the representation varies in $C_i$ (and its complexification).
\end{theorem}

\begin{remmark}The integer $n$ correspond to half the dimension of any orbit in $\Omega_i$.\end{remmark}

\begin{example}For the Heisenberg group $G = H_{2n+1}$, we consider the dual coordinates $(x_1,\cdots,x_n,y_1,\cdots,y_n,z)$ on $\g^*$ with respect to the usual basis on $\g$. The stratification is then given by $\Omega_1 =  \{(x,y,z) \ / \ z \neq 0\}$ and $\Omega_2 = \{(x,y,0) \ / \ (x,y)\neq (0,0)\}$. When taking the quotient by the coadjoint action we get $\Lambda_1 \cong \R^*$ and $\Lambda_2 \cong \R^{2n}$. The $\R^*_+$ action is given by dilations on each components: $\delta_{\lambda|\Lambda_1} = \lambda^2$ and $\delta_{\lambda|\Lambda_2} = \lambda$.\end{example}

Given a graded Lie group, we can use Pedersen's stratification to decompose its $C^*$-algebra. Let us denote by $V_1 \subset \cdots \subset V_d = \widehat{G}$ Pedersen's stratification of the unitary dual, i.e. $V_i = \faktor{W_i}{G}$ where $W_i$ is the union of the first $i$-th strata in Pedersen's stratification of $\g^*$. We have $V_i\setminus V_{i-1} =\Lambda_i$. Corresponding to these open subsets are ideals $J_i := \bigcap_{\pi \in \widehat{G}\setminus V_i}\ker(\pi)$ of $C^*(G)$. We get an increasing sequence of ideals:
$$\{0\} = J_0 \triangleleft J_1 \triangleleft \cdots \triangleleft J_d = C^*(G)$$
and by construction $\widehat{\faktor{J_i}{J_{i-1}}} = \Lambda_i$. The subquotients thus have Hausdorff spectrum and therefore are continuous fields of $C^*$-algebras over $\Lambda_i$ according to \cite{Nilsen}. Using the result of Lipsman-Rosenberg stated above we have:

\begin{corollary}\label{SubquotientsCompact}For each $1 \leq i \leq d$ we have:
$$\faktor{J_i}{J_{i-1}} \cong \CCC_0(\Lambda_i,\mathcal{K}_i).$$
Here $\mathcal{K}_i$ denotes the algebra of compact operators on a separable Hilbert space (of dimension $1$ for $i = d$ and infinite dimensional otherwise).
\end{corollary}

\begin{remmark}Compare this decomposition with the one given by Puckansky's stratification used in \cite{Ewert}. 
The later still gives subquotients that are continuous fields of $C^*$-algebras over the spectrum but we cannot ensure their triviality. 
Indeed, this triviality follows from the results of \cite{LipRos} which only apply to Pedersen's stratification. 
The difference between the two stratifications is not easy to see for groups of low dimension and depth (e.g. the Heisenberg group). 
Examples of graded groups for which the two stratifications differ are given at the end of \cite{Bonnet}.
\end{remmark}

We end this section with the definition and properties of quasi-norms on graded Lie groups. They will serve two purposes in the following. The first one will be to realize elements of $C^*(G)$ as symbols for the filtered calculus. It will also provide a section for the quotient map $\widehat{G}\setminus \{0\}\to \faktor{\widehat{G}\setminus \{0\}}{\R^*_+}$. It is the analog of a norm in the commutative case except now we want it to be homogeneous with respect to the inhomogeneous family of dilations.

\begin{definition}
A homogeneous quasi-norm on a Lie group $G$ is a continuous function $|\cdot | \colon G \to \R_+$ such that:
\begin{itemize}
\item $|g| = 0$ if and only if $g$ is the neutral element of $G$
\item $|\cdot|$ is homogeneous of degree 1, i.e. $\forall \lambda > 0, \forall g \in G, |\delta_{\lambda}(g)| = \lambda |g|$.
\end{itemize}
\end{definition}

Such a function always exists and any two of them are equivalent. Now given a homogeneous quasi-norm $|.|$ we can induce a similar function on $\g$ through the exponential map. By duality we obtain a similar function on $\g^*$ for which we keep the same notation. We then obtain a map on $\widehat{G}$ through Kirillov homomorphism and setting $|\Ad^*(G)\cdot \xi| := \inf_{g\in G}|\Ad^*(g)(\xi)|$. The function $|\cdot|\colon \widehat{G} \to \R_+$ is still continuous. It is homogeneous of degree 1 with respect to the dilations on $\widehat{G}$ and vanishes only at the class of the trivial representation. We then obtain

\begin{lemma}
Let $|\cdot|$ be a homogeneous quasi-norm on $G$. The canonical map 
$$\left\lbrace\pi \in \widehat{G} \ / \ |\pi| = 1\right\rbrace \to \faktor{\widehat{G}\setminus \{0\}}{\R^*_+}$$
is a homeomorphism.
\end{lemma}

\section{Osculating group bundle of a filtered manifold}

Calculus on graded nilpotent Lie groups appears naturally when one studies operators on filtered manifolds. In this section we describe the basic geometric features of such manifolds.

\begin{definition}A filtered manifold is a manifold $M$ equipped with a filtration of its tangent bundle by subbundles $\{0\} = H^0 \subset H^1 \subset \cdots \subset H^r = TM$ with the following condition on the Lie brackets:
$$\forall i,j, \left[\Gamma(H^i),\Gamma(H^j)\right] \subset \Gamma(H^{i+j}) $$
with $H^k = TM$ for $k \geq r$.
\end{definition}

Given a filtered manifold $(M,H)$ we can associate to it a bundle of graded nilpotent Lie groups called the osculating group bundle. It replaces the tangent bundle and captures its filtered structure. We start by defining the associated bundle of Lie algebras. Denote by $\ttt_HM = H^1 \oplus \faktor{H^2}{H^1} \oplus \cdots \oplus \faktor{H^r}{H^{r-1}}$. We endow it with a Lie algebroid structure with trivial anchor and the following bracket: if $X \in \Gamma(H^i), Y \in \Gamma(H^j)$ we have $[X,Y] \in H^{i+j}$ and its class $[X,Y]\mod \Gamma(H^{i+j-1})$ only depends on the class of $X\mod \Gamma(H^{i-1})$ and the class of $Y\mod \Gamma(H^{j-1})$. We thus have defined a Lie bracket
$$[\cdot,\cdot] \colon \Gamma(\ttt_HM) \wedge \Gamma(\ttt_HM) \to \Gamma(\ttt_HM).$$
This Lie Bracket is $\CCC^{\infty}(M)$-bilinear, indeed let $f \in \CCC^{\infty}(M), X \in \Gamma(H^i), Y \in \Gamma(H^j)$, we have
$$[fX,Y] = f[X,Y] - \diff f(Y)X \equiv f[X,Y] \mod \Gamma(H^{i+j-1}).$$
This implies that the algebroid structure on $\ttt_HM$ is from a smooth family of Lie algebras. These algebras are graded by construction. 

\begin{definition}Let $(M,H)$ be a filtered manifold. We denote by $T_HM$ the osculating groupoid. It is the bundle of (connected, simply connected) Lie groups integrating $\ttt_HM$ obtained by the Baker-Campbell-Hausdorff formula.\end{definition}

The term \begin{it}osculating\end{it} goes back to \cite{FollandSteinEstimate} with the Heisenberg group arising from a CR-structure. This notion has since then been extended to arbitrary filtered manifolds, see e.g. \cite{van_erp_2017}.

\begin{example}For the trivial filtration $H^1 = TM$ we get $T_HM = TM$ as a bundle of abelian groups (locally trivial). \end{example}

\begin{example}Let $M$ be a contact manifold with contact structure $H \subset TM$. Then on each fiber $T_HM$ has the structure of a Heisenberg group, moreover this bundle is locally trivial (it is trivialized in Darboux coordinates).\end{example}

\begin{example}
    If \(G\) is a semi-simple Lie group, \(P \subset G\) a parabolic subgroup and \(M\) a manifold with \((G,P)\)-geometry then \(M\) has a natural structure of a filtered manifold for which the osculating groupoid is a locally trivial bundle of groups \(N_-\) where \(N_-\) is the nilpotent group opposed to \(P\), see \cite{dave2017graded}.
\end{example}

Despite the apparent simplicity of these examples, the general case can be more complicated : the group structure might change abruptly between fibers and even if all the fibers are isomorphic there might be no analog of the Darboux theorem.
We can encompass these kind of bundles and graded Lie groups by dealing with arbitrary smooth families of graded Lie groups over smooth manifolds (non necessarily locally trivial bundles). These are Lie groupoids \(\mathcal{G} \rr M\) with identical source and range maps and with connected and simply connected fibers. Their algebroid \(\Lie(\G) \to M\) is a graded vector bundle and the Lie algebra structures on the fibers are graded for this grading.

Despite the absence of local theory for such bundles, our strategy will be to reduce our study to graded Lie groups. Group bundles are an example of Lie groupoids. Associated to them there is thus a $C^*$-algebra. Because the fibers of the bundle are nilpotent Lie groups the bundles are amenable groupoids so we will not make any difference between the maximal and reduce $C^*$-algebras. The end of this section is devoted to the description of this kind of algebra in the light of our strategy.

\begin{example}If $\mathcal{G} = G_0 \times M$ is a trivial Lie group bundle over $M$ then $C^*(\mathcal{G}) = C^*(G_0) \otimes \CCC_0(M)$.\end{example} 

\begin{proposition}\label{ContField}Let $\mathcal{G} \to M$ be a smooth family of graded Lie groups. The algebra $C^*(\mathcal{G})$ is a continuous field of $C^*$-algebras over $M$, the fiber at $x \in M$ is equal to the group $C^*$-algebra $C^*(\mathcal{G}_x)$.\end{proposition}
\begin{proof}
This is an easy consequence of the fact that $G$ is a smooth (hence continuous) family of group(oid)s over $M$ in the sense of \cite{LandRam}.
\end{proof}

\begin{definition}
    If \(\G \to M\) is a smooth family of graded groups we define \(C^*_0(\G)\) to be the ideal in \(C^*(\G)\) corresponding to the open subset \(\widehat{\G}\setminus M\) where \(M\) is seen as the zero section in \(\Lie(\G)^*\).
\end{definition}

\begin{corollary}\label{ContFieldCstarZero}
    If \(\G \to M\) is a smooth family of graded groups then $C^*_0(\mathcal{G})$ is a continuous field of $C^*$-algebras over $M$, the fiber at $x \in M$ is equal to the $C^*$-algebra $C^*_0(\mathcal{G}_x)$.
\end{corollary}

\section{Stratification of coadjoint orbits for a family of graded groups}\label{Section:Stratification for Family}

In this Section we extend Pedersen's stratification to general smooth families of graded groups (in particular, osculating groupoids). This allows us to get a similar result as in Corollary \ref{SubquotientsCompact} for the \(C^*\)-algebra of a smooth family of graded groups (i.e. solvability). It is however not crucial in the proofs later on since they can be reduced to the case of a single group using Corollary \ref{ContFieldCstarZero} and similar result for the algebra of principal symbols (see Section \ref{Section:Symbols}).

Let us first recall the construction of Pedersen's strata in the case of a group. Consider \(G\) a connected, simply connected nilpotent group, and \(\g\) its Lie algebra. Let \(X_1,\cdots, X_m\) be a Jordan-Hölder basis of \(\g\). Given an element \(\xi \in \g^*\) and a subalgebra \(\mathfrak{h}\subset \g\) we define:
\[\mathfrak{h}(\xi) = \left\lbrace X \in \mathfrak{h}, \forall Y \in \mathfrak{h}, \lag \xi, [X,Y]\rag = 0\right\rbrace.\]
Consider \(\g_1\triangleleft \cdots \triangleleft \g_m = \g\) the nested sequence of ideals defined by the Jordan-Hölder basis, i.e. \(\g_j = \operatorname{Span}\left(X_{n-j+1},\cdots,X_n\right), j\geq 0\).

\begin{definition}
    The jump indices of an element \(\xi \in \g^*\) are the sets \(J^k_{\xi}, 1\leq k \leq m\) where:
    \[J^k_{\xi} = \left\lbrace 1\leq j \leq k, X_j \notin \g_{j-1} + \g_k(\xi) \right\rbrace \in \mathcal{E},\]
    and \(\mathcal{E}\) is the powerset of \(\{1,\cdots,m\}\). We write 
    \[\mathcal{J}_{\xi} := (J^1_{\xi},\cdots,J^m_{\xi}) \in \mathcal{E}^m.\]
\end{definition}

Since for a given ideal \(\mathfrak{h}\subset \g\), \(\xi \in \g^*\) and \(g\in G\) we have \(\mathfrak{h}(\xi) = \mathfrak{h}(\Ad^*(g)\xi)\), then the jump indices are identical along a given co-adjoint orbit.

If \(\varepsilon \in \mathcal{E}^m, \Omega_{\varepsilon} := \{\xi \in \g^*, \mathcal{J}_{\xi} = \varepsilon\}\) is a stratum of Pedersen's stratification. The ordering on strata corresponds to the ordering on \(\mathcal{E}^m\) given by lexicographic order and inclusion, see \cite{PedersenQuant}.

The stratification obviously extends to families of graded groups that are trivial bundles, i.e. \(\mathcal{G} = G_0 \times M \to M\). We can also extend the stratification to locally trivial bundles of graded groups using the following lemma:

\begin{lemma}
    Let \(\g\) be a nilpotent group endowed with a Jordan-Hölder basis \((X_1,\cdots,X_m)\) and \(\varphi \colon \g \to \g\) be a Lie algebra automorphism. Then the basis \((\varphi(X_1),\cdots,\varphi(X_m))\) is also a Jordan-Hölder basis. If we write \(\tilde{\mathcal{J}}\) for the jump indices corresponding to the second basis then:
    \[\forall \xi \in \g^*, \mathcal{J}_{\tsp\varphi(\xi)} = \tilde{\mathcal{J}}_{\xi}.\]
    In particular, \(\tsp\varphi\) sends the Pedersen strata induced by the second basis to the strata of the one induced by the first basis (preserving the stratification ordering).
\end{lemma}
\begin{proof}
    It is clear that the second basis is also a Jordan-Hölder basis. Write \(\tilde{\g}_1\triangleleft \cdots \triangleleft \tilde{\g}_m = \g\) for the sequence of nested ideals corresponding to the second basis. Then \(\varphi(\g_j) = \tilde{\g}_j\) for all \(j\). Moreover if \(\xi \in \g^*\) and \(1\leq k \leq m\), we have \(\varphi\left(\g_k(\tsp\varphi(\xi))\right) = \tilde{\g}_k(\xi)\) so we get the result on jump indices.
\end{proof}

\begin{remmark}
    Given two arbitrary Jordan-Hölder basis for a Lie algebra \(\g\), the linear automorphism sending one to the other might fail to preserve the Lie algebra structure. In particular, the previous result doesn't imply that there is only one possible Pedersen stratification. However, if one fixes a Jordan-Hölder basis, then we can transport it with any Lie algebra automorphism and the strata in Pedersen's stratification will be preserved (even though the automorphism doesn't preserve the flag structure induced by the Jordan-Hölder basis).
\end{remmark}

\begin{corollary}\label{Stratification Loc Triv}
    Let \(\mathcal{G} \to M\) be a locally trivial bundle of graded groups with group structure \(G\). There exists a filtration \(\emptyset = \mathcal{V}_0 \subset \mathcal{V}_1 \subset \cdots \subset \mathcal{V}_d = \hat{\mathcal{G}}\setminus M\) by open subsets such that each \(\Omega_j = \mathcal{V}_j \setminus \mathcal{V}_{j-1}\) is locally compact Hausdorff, \(\R^*_+\)-invariant and the action \(\R^*_+ \act \Omega_j\) is free and proper. 
    
    The sets \(\mathcal{V}_j\) and \(\Omega_j\) are fiber bundles over \(M\) with respective fiber \(V_j\) and \(\Lambda_j\) which are taken from Corollary \ref{Stratification} applied to \(G\).
\end{corollary}

\begin{proof}
    We fix a a Jordan-Hölder basis for the Lie algebra of \(G\). We can them define the spaces \(\mathcal{V}_j\) locally and glue them back using the previous lemma.
\end{proof}

\begin{corollary}\label{SubquotientsCompact Loc Triv}
    In the same conditions as in the previous Corollary and With the same notations, the \(C^*\)-algebra \(C^*_0(\mathcal{G})\) decomposes into a sequence of nested ideals: 
    \[\{0\} = J_0 \triangleleft J_1 \triangleleft \cdots \triangleleft J_d = C^*_0(\mathcal{G}),
    \]
    with \(\widehat{J_k} = \mathcal{V}_k\).
    Then for each \(1\leq k \leq d\) we have 
    \[\faktor{J_k}{J_{k-1}} \cong \CCC_0(\Omega_k,\mathcal{K}_k).\]
    Here \(\mathcal{K}_i\) denotes the algebra of compact operators on a separable Hilbert space (of dimension \(1\) for \(k = d\) and infinite dimensional otherwise).
\end{corollary}

\begin{proof}
    The ideals \(J_k\) are still continuous fields of \(C^*\)-algebras over \(M\) so we can apply the results of Corollary \ref{SubquotientsCompact} locally and glue them using Corollary \ref{Stratification Loc Triv}.
\end{proof}

\begin{example}
    Let \((M,H)\) be a contact manifold and consider its osculating groupoid \(T_HM\). The groupoid \(T_HM\) is a locally trivial bundle of Heisenberg groups. We fix a Jordan-Hölder basis for the Heisenberg Lie algebra by taking any basis of \(\R^{2n}\) and then a basis element of the center. Then \(\Omega_1 \cong \left(\faktor{TM}{H}\right)^*\setminus M\) and \(\Omega_2 = H^*\setminus M\).
\end{example}

We finally would like to extend these results to an arbitrary family of graded groups. Let \(\mathcal{G} \to M\) be a smooth family of graded groups. Despite the group structure on the fibers not being locally trivial, the family \(\Lie(\G) \to M\) is still a vector bundle.

\begin{proposition}\label{Universal bundle}
    There exists a locally trivial bundle of graded groups \(\mathcal{F} \to M\) and a smooth submersion \(\Phi \colon \F \to \G\) which is a bundle map, preserves the graded group structure on the fibers and is onto.
\end{proposition}

\begin{proof}
    Locally we trivialize the graded vector bundle \(\Lie(\G)\) over a subset \(U\subset M\) and get a basis of each component of the grading. If we take them in ascending order, this gives a local frame \(X_1,\cdots,X_m\) which is a Jordan-Hölder basis on each fiber. Let \(\mathfrak{f}\) be the universal graded Lie algebra with same depth as \(Lie(\G)\) and generators corresponding to the \(X_i\)'s in the same degree. This defines a map \(U\times \mathfrak{f} \to \Lie(\G)_{|U}\). We can lift this map to a map \(\Phi_U\colon U \times F \to \G_{|U}\) where \(F\) is the connected simply connected group integrating \(\mathfrak{f}\). If we change the local frame this induces a automorphism of \(\mathfrak{f}\) on each fiber. Therefore this local construction can be glued back to a locally trivial bundle of graded groups \(\F \to M\) with fiber \(F\). The previous map then globalizes to \(\Phi \colon \F \to \G\) and has the desired properties. 
\end{proof}

We now also consider this locally trivial bundle \(\F \to M\). We can apply Corollaries \ref{Stratification Loc Triv} and \ref{SubquotientsCompact Loc Triv} to this bundle. Notice that since \(C^*_0(\G)\) is a quotient of \(C^*_0(\F)\), then the former is also solvable by taking the quotients of the nested ideals of \(C^*_0(\F)\). The locally compact spaces appearing in the subquotients are closed subsets of the \(\Omega_j\)'s corresponding to \(\F\). In order to identify them, we need to understand what the map \(\tsp\diff\Phi \colon \Lie(\G)^* \to \Lie(\F)^*\) does to the Pedersen strata. This is described in the following proposition. We state this result for single groups for simplicity but since we have local Jordan-Hölder basis for \(\Lie(\G)\), this clearly generalizes to the case at hand.

\begin{proposition}\label{Strata Universal Nilp}
    Let \(G\) be a graded Lie group with Lie algebra \(\g\) endowed with a Jordan-Hölder basis  \(X_1,\cdots,X_m\) (coming from the grading). Let \(\mathfrak{f}\) be the universal graded Lie algebra with same depth as \(\g\) and generated by the \(X_j\)'s in the same degree. Let \(\varphi \colon \mathfrak{f} \to \g\) be the Lie algebra homomorphism sending \(X_j\) to \(X_j\) for each \(1\leq j \leq m\). There exists a Jordan-Hölder basis \(Z_1, \cdots Z_M\) for \(\mathfrak{f}\) for which the elements are either equal to some \(X_j\) (and they all appear once, in order) or to an element of \(\ker(\varphi)\).

    Write \(u \colon \{1,\cdots,m\} \to \{1,\cdots,M\}\) for the map defined by \(Z_{u(j)} = X_j, 1\leq k \leq m\). For \(1\leq k \leq M\), let \(1\leq \bar{k}\leq m\) be the biggest index for which \(u(\bar{k}) \leq k\). Then:
    \[\forall \xi \in \g^*, \forall 1\leq k\leq M, J^k_{\tsp\varphi(\xi)} = u(J_{\xi}^{\bar{k}}).\]

    In particular \(\tsp{\varphi}\) sends Pedersen strata to Pedersen strata injectively and preserving their ordering. In other words, the pre-image by \(\tsp\varphi\) of a Pedersen stratum in \(\mathfrak{f}^*\) is either empty or a Pedersen stratum of \(\mathfrak{g}^*\).
\end{proposition}

\begin{proof}
    We consider the \(X_j\)'s as elements of both \(\f\) and \(\g\). Then we can complete the Jordan-Hölder basis of \(\f\) the following way: after all the \(X_j\)'s of some fixed degree \(k\) we concatenate a basis elements of degree \(k\) in the kernel of the map \(\varphi\colon \mathfrak{f} \to \g\). Since \(\mathfrak{f}\) is a graded Lie algebra, this is a Jordan-Hölder basis. We denote this Jordan-Hölder basis \(Z_1,\cdots,Z_M\).

    Let \(\xi \in \g^*\) and \(1\leq k \leq M\). We have \(\varphi(\f_k) = \g_{\bar{k}}\) so a quick computation yields 
    \[\f_k(\tsp\varphi(\xi)) = \varphi^{-1}(\g_{\bar{k}})\cap \f_k.\]
    Let \(1\leq \ell \leq k\): 
    \begin{itemize}
        \item if \(\ell \notin \im(u)\) then \(Z_{\ell}\in \ker(\varphi)\cap \f_k \subset \f_k(\tsp\varphi(\xi))\) so \(\ell \notin J^k_{\tsp\varphi(\xi)}\).
        \item if \(\ell = u(\bar{\ell})\) then \(Z_{\ell} = X_{\bar{\ell}} \in \f_k(\tsp\varphi(\xi)) + \f_{\ell-1}\) if and only if \(X_{\bar{\ell}} \in \g_k(\xi) + \g_{\overline{\ell-1}}\). But since \(\ell = u(\bar{\ell})\) then \(\overline{\ell-1} = \bar{\ell}-1\) and we get the result.\qedhere
    \end{itemize}
\end{proof}

\begin{corollary}\label{Stratification General family}
    Let \(\G \to M\) be a smooth family of graded groups. 
    %Let \(\F \to M\) be the locally trivial bundle of graded groups with fiber \(F\), universal graded group of same depth as any fiber of \(\G\). Let \(\Phi \colon \F \to \G\) be the submersion of Proposition \ref{Universal bundle}. Let us fix a Jordan-Hölder basis for \(\Lie(F)\) as in Proposition \ref{Strata Universal Nilp}. 
    There exists a filtration \(\emptyset = \mathcal{V}_0 \subset \mathcal{V}_1 \subset \cdots \subset \mathcal{V}_r = \hat{\mathcal{G}}\setminus M\) by open subsets such that each \(\Omega_j = \mathcal{V}_j \setminus \mathcal{V}_{j-1}\) is locally compact Hausdorff, \(\R^*_+\)-invariant and the action \(\R^*_+ \act \Omega_j\) is free and proper. The sets \(\mathcal{V}_j\) and \(\Omega_j\) are fibered over \(M\). The fibers over a given point \(x\in M\) form a Pedersen stratification of \(\widehat{\G_x}\) (some \(\Omega_{j|x}\) might be empty).
\end{corollary}

\begin{proof}
    Let \(\F \to M\) be the locally trivial bundle of graded groups with fiber \(F\), universal graded group of same depth as any fiber of \(\G\). Let \(\Phi \colon \F \to \G\) be the submersion of Proposition \ref{Universal bundle}. Let us fix a Jordan-Hölder basis for the typical fiber of \(\Lie(F)\) as in Proposition \ref{Strata Universal Nilp}. Then we pullback the stratification of \(\widehat{\F}\setminus M\) given by Corollary \ref{Stratification Loc Triv} to \(\widehat{\G}\setminus M\). The fibers are then identified over each point with a Pedersen stratification for \(\widehat{\G_x}\setminus \{0\}\) using Proposition \ref{Strata Universal Nilp}.
\end{proof}

\begin{corollary}\label{SubquotientsCompact General family}
    Let \(\G \to M\) be a smooth family of graded groups and consider a filtration of \(\hat{\G}\setminus M\) as in the previous Corollary. The \(C^*\)-algebra \(C^*_0(\mathcal{G})\) decomposes into a sequence of nested ideals: 
    \[\{0\} = J_0 \triangleleft J_1 \triangleleft \cdots \triangleleft J_d = C^*_0(\mathcal{G}),
    \]
    with \(\widehat{J_k} = \mathcal{V}_k\).
    Then for each \(1\leq k \leq d\) we have 
    \[\faktor{J_k}{J_{k-1}} \cong \CCC_0(\Omega_k,\mathcal{K}_k).\]
    Here \(\mathcal{K}_k\) denotes the algebra of compact operators on a separable Hilbert space (of dimension \(1\) for \(k = d\) and infinite dimensional otherwise).
\end{corollary}

\begin{proof}
    With the same notations as before, \(C^*_0(\G)\) is a quotient of \(C^*_0(\F)\) which is solvable. The algebra \(C^*_0(\G)\) is then solvable as well and we have identified the locally compact spaces giving the spectra of subquotients using Proposition \ref{Strata Universal Nilp}.
\end{proof}

\begin{remmark}
    If for a given \(1\leq k \leq d\) and \(x\in M\) the set \(\Omega_{k|x}\) is empty, this means that an element \(f\in \faktor{J_k}{J_{k-1}}\) seen as a function on \(\Omega_k\) vanishes near the \(x\)-fiber in the following sense:
    \[\lim_{p_k(\lambda) \to x} f(\lambda) = 0,\]
    with \(p_k\colon \Omega_k \to M\) being the fiber map in Proposition \ref{Stratification General family}.
\end{remmark}

The next example showcases this vanishing aspect:

\begin{example}
    Let \(\G = \R^3\times \{0\} \sqcup H_3 \times (0;1] \to [0;1]\). We can see it as integrating the bundle \(\R^3 \times [0;1] \to [0;1]\) where the Lie bracket on the \(t\)-fiber is equal to \(t\) times the one of the Heisenberg Lie algebra (so we indeed get an abelian Lie algebra for \(t = 0\)).
    If we take a basis \(X,Y,Z\) of the Lie algebra then we get \(\mathfrak{f} = \lag X,Y,[X,Y],Z\rag \cong \mathfrak{h}_3\oplus \R\). Then we get \(\Omega_1 = \R^* \times (0;1]\) and \(\Omega_2 = \R^3 \sqcup \R^2 \times (0;1] = \{(x,y,z,t) \in \R^3 \times [0;1], zt = 0\}.\)
    In particular we get \(\Omega_{1|t=0} = \emptyset\). This makes sense because \(\G_0\) is abelian so there can only be one Pedersen stratum and it has to be the top one.
\end{example}

\section{Symbols in the filtered calculus}\label{Section:Symbols}

Let $\mathcal{G}$ be a Lie groupoid. Denote by 
\begin{align*}
    \hd &= \Lambda^{1/2}(\ker(\diff s))\otimes \Lambda^{1/2}(\ker(\diff r)) \\
        &\cong r^*\Lambda^{1/2}(\Lie(\G)) \otimes s^*\Lambda^{1/2}(\Lie(\G))
\end{align*}
the bundle of half densities on $\mathcal{G}$. Write $\CCC^{\infty}_c(\mathcal{G}, \hd)$ for the algebra of sections with the convolution product $f\ast g (\gamma) = \int_{\gamma_1\gamma_2 = \gamma} f(\gamma_1)g(\gamma_2)$. Equivalently, one can also replace half-densities by a choice of a Haar system for $\mathcal{G}$ which would correspond to a trivialization of $\Lambda^{1/2}(\Lie(\G))$, both yield the same convolution algebra for $\mathcal{G}$.

For distributions: $\mathcal{D}'(\mathcal{G}, \hd)$ (resp. $\mathcal{E}'(\mathcal{G}, \hd)$) will denote the topological dual of $\CCC^{\infty}(\mathcal{G},\Omega^{-1/2} \otimes \Omega^1_{\mathcal{G}})$ (resp. $\CCC^{\infty}(\mathcal{G},\Omega^{-1/2} \otimes \Omega^1_{\mathcal{G}})$) where $\Omega^1_{\mathcal{G}}$ is the bundle of 1-densities obtained from $T\mathcal{G}$. This choice is made so that $\CCC^{\infty}(\mathcal{G},\hd) \hookrightarrow \mathcal{D}'(\mathcal{G}, \hd)$ (resp. $\CCC^{\infty}_c(\mathcal{G},\hd) \hookrightarrow \mathcal{E}'(\mathcal{G}, \hd)$).

Let us recall the definition of symbols in the context of filtered calculus. In the following, $\pi \colon \G \to M$ denotes a smooth family of graded Lie groups. The family of inhomogeneous dilations $(\delta_{\lambda})_{\lambda > 0}$ defines by pullback a family of algebra automorphisms $(\delta_{\lambda}^*)_{\lambda > 0}$ of $\CCC^{\infty}_c(\G, \hd)$. Note that this is the pullback of half-densities which takes the jacobian into account. If a choice of Haar system is made, the corresponding action on $\CCC^{\infty}_c(\G)$ is given by $(\lambda^n\delta_{\lambda}^*)_{\lambda > 0}$ (the pullback of functions will not preserve the convolution product). Here $n$ denotes the homogeneous dimension of $G$, i.e. $n = \sum_{i\geq 1} i \rk(\g_i)$.

\begin{definition}A symbol of order $m \in \CC$ in $\G$ is a distribution $u \in \mathcal{D}'(\G, \hd)$ satisfying
\begin{itemize}
\item $u$ is properly supported, i.e. $\pi \colon \supp(u) \to M$ is a proper map.
\item $u$ is transversal to $\pi$ (in the sense of \cite{androulidakis2009holonomy}, also called \(\pi\)-fibered in \cite{van_erpyunken}), i.e. $\pi_*(u) \in \CCC^{\infty}(M)$.
\item $\forall \lambda \in \R^*_+, \delta_{\lambda *}u -\lambda^m u \in \CCC^{\infty}_p(\G)$.
\end{itemize}
$S_p^m(\G)$ denotes the set of these distributions and $S^*_p(\G) = \bigcup_{m \in \Z}S_p^m(\G)$.
The subspace of compactly supported symbols will be denoted by $S_c^m(\G)$. We will write $S^m(\G)$ when the statements apply for both $S^m_p(\G)$ and $S^m_c(\G)$.
\end{definition}

The second condition implies that $u$ corresponds to a smooth family of distributions $u_x \in \mathcal{D}'(\G_x,\hd)$ (see \cite{lescuremanchonvassout} for a more precise statement). The first condition then implies that the support of each $u_x$ is a compact subset of the fiber $G_x$ for every $x \in M$, i.e. $u_x \in \mathcal{E}'(\G_x,\hd)$. The third condition implies that $u$ is smooth outside of the zero unit section $M \subset \G$, it also gives the asymptotic behavior of $u$ near $M$.

The symbols used here are called quasi-homogeneous. They correspond to principal symbols of pseudodifferential operators in the filtered calculus up to $\CCC^{\infty}_p(\G)$ functions. For compactly supported symbols, the quasi-homogeneity condition can equivalently be stated $\mod \CCC^{\infty}_c(\G)$.

The push-forward of distributions is obtained by duality from the pullback of half-densities. To understand what it does on functions seen as distributions first note that $\Omega^1_{\G} \otimes \Omega^{-1/2} \cong (r^* \otimes s^*)(\Omega_M^{1/2})$. On elements $f \in \CCC^{\infty}(\mathcal{G},\Omega^{-1/2} \otimes \Omega^1_{\mathcal{G}})$ seen as elements of $\mathcal{D}'(\mathcal{G}, \hd)$ we then have:
$$\delta_{\lambda *}f = \delta_{\lambda^{-1}}^*f.$$

\begin{remmark} As we use properly supported distributions we do the same for functions: $\CCC^{\infty}_p(\cdot)$ denotes the space of properly supported functions on a groupoid (if $G$ is a groupoid, $X\subset G$ is proper if $s_{|X}$ and $r_{|X}$ are proper maps). With these notations we see that $\CCC^{\infty}_p(G,\hd)$ embeds into the space of properly supported fibered distributions on $G$ (an arbitrary Lie groupoid).
\end{remmark}

We recall some useful results on filtered calculus, for more details and proofs we refer to \cite{van_erpyunken, mohsen2020index, dave2017graded} and their references:

\begin{lemma}Let $u\in S^{m_1}(\G), v\in S^{m_2}(\G)$ then $u\ast v \in S^{m_1+m_2}(\G)$ and $u^* \in S^{m_1}(\G)$. This result can be extended to $m_1 = -\infty$ or $m_2 = -\infty$ with $S^{-\infty}_{c/p}(\G) = \CCC^{\infty}_{c/p}(\G,\hd)$.\end{lemma}

Let $u \in S^*(\G)$, for each $x \in M$ we get a convolution operator 
$$\begin{matrix}
\op(u_x) & \colon & \CCC^{\infty}_c(\G_x,\hd) & \longrightarrow     & \CCC^{\infty}_c(\G_x,\hd) \\
         &        &  f                   & \longmapsto & u_x \ast f &
\end{matrix}.$$
The transversality condition then allows us to "glue" these operators to obtain\\
$\op(u) \colon \CCC^{\infty}_c(\G,\hd) \to \CCC^{\infty}_c(\G,\hd)$ with $\op(u)(f)_{|\pi^{-1}(x)} = \op(u_x)(f_{|\pi^{-1}(x)})$.

\begin{lemma}Let $u,v \in S^*(\G)$ and $f,g \in \CCC^{\infty}_c(\G,\hd)$ then
\begin{itemize}
\item $\op(u)(f\ast g) = \op(u)(f)\ast g$
\item $\op(u\ast v) = \op(u) \circ \op(v)$
\end{itemize}
\end{lemma}

We now state the Rockland condition and its consequences on the symbolic calculus. Rockland condition replaces the usual ellipticity condition. Moreover, when the symbols yield operators on $M$ (i.e. for $\G = T_HM$, see \cite{van_erpyunken}) the operators whose symbols satisfy Rockland's condition are hypoelliptic and admit parametrices (in the filtered calculus). This condition was first introduced by Rockland in \cite{Rockland} for differential operators on Heisenberg groups. It was then extended to manifolds with an osculating Lie group of rank at most two by Helffer and Nourrigat in \cite{HelfferNourrigat} and Lie groups with dilations in \cite{Groupsdilation}. See also the recent advances in \cite{AndroulidakisMohsenYunken} in a very broad setting generalizing the results of Helffer and Nourrigat.

\begin{definition}A symbol $u \in S^*(\G)$ satisfies the Rockland condition at $x \in M$ if there exist a compact set $K \subset \widehat{\G_x}$ on the dual space of irreducible representations such that for all $\pi \notin K$, $\pi(\op(u))$ and $\pi(\op(u^*))$ are injective. The symbol satisfies the Rockland condition if it satisfies it at every point $x \in M$ ($u$ is then also said to be Rockland).
\end{definition}

As in the classical case, Rockland condition corresponds exactly to the existence of parametrices. It was proved in \cite{Groupsdilation} for trivial bundles and this result was used in \cite{dave2017graded} for the proof in the general case.

\begin{theorem}[\cite{dave2017graded}]\label{RocklandEq}Let $u \in S^k_p(\G)$ then there is an equivalence :
\begin{itemize}
\item[1] $u$ satisfies the Rockland condition
\item[2] there exists $v \in S^{-k}_p(\G)$ such that $u\ast v -1, v\ast u -1 \in \CCC^{\infty}_p(\G,\hd)$ 
\end{itemize}
\end{theorem}
A proof can be found in \cite[Lem. 3.9, Lem. 3.10]{dave2017graded}, using arguments originating in \cite{Groupsdilation} (which corresponds to the case of trivial bundles).

\begin{theorem}\label{OpSymbols}Let $u \in S^k_c(G)$ then
\begin{itemize}
\item If $\Re(k)\leq 0$ then $\op(u)$ extends to an element of the multiplier algebra $\m(C^*(G))$
\item If $\Re(k)<0$ then $\op(u)$ extends to an element of $C^*(G)$
\item If $k>0$, $M$ is compact and $u$ satisfies the Rockland condition then $\op(u)$ extends to an unbounded regular operator $\overline{\op(u)}$ on $C^*(G)$ (viewed as a $C^*$-module over itself) and $\overline{\op(u)}^* = \overline{\op(u^*)}$
\end{itemize}
\end{theorem}
See \cite[Thm. A.8]{mohsen2020index} for a detailed proof.

\begin{definition}A symbol $u \in S^*(T_HM)$ is Rockland (i.e "elliptic" in the filtered calculus) if it satisfies the Rockland condition.\end{definition}

Theorem \ref{OpSymbols} allows to complete the algebra of symbols. Let $\bar{S}^0_0(G) \subset \m(C^*(G))$ be the $C^*$-closure of $S^0_c(G)$ (we identify the algebra of order 0 symbols and its image by $\overline{\op}$).

\begin{definition}We denote by $\Sigma^k_{p/c}(G) = \faktor{S^k_{p/c}(G)}{S^{-\infty}_{p/c}(G)}$ of order $k \in \CC$ the set of principal symbols in the filtered calculus. For $k = 0$ we also have $\Sigma(G) = \faktor{\bar{S}^0_0(G)}{C^*(G)}$ for the $C^*$-algebra of principal symbols. The algebra $\Sigma(G)$ is a completion of $\Sigma^0_c(G)$.
\end{definition}

\begin{example}Let $\mathcal{U}(\g)$ be the enveloping algebra of the Lie algebra bundle $\g$. The inhomogeneous dilations $(\delta_{\lambda})_{\lambda\in \R^*_+}$ extend to algebra automorphisms of $\mathcal{U}(\g)$. Denote by $\mathcal{U}_k(\g)$ the subset of elements homogeneous of order $k$ for the dilations. Elements of the enveloping algebra decompose as sums of homogeneous elements, i.e. $\mathcal{U}(\g) = \bigoplus_{k \geq 0} \mathcal{U}_k(\g)$. Any element $D \in \Gamma(M,\mathcal{U}_k(\g))$ can be seen as a family of right-invariant differential operators on the fibers of $G$ and thus defines an element $D \in S^k_p(G)$.
\end{example}

Let $k > 0$ and $D \in \Gamma(M,\mathcal{U}_k(\g))$ be a Rockland differential operator. Assume $D$ to be self-adjoint and non-negative and $M$ to be a compact manifold. The third point of Theorem \ref{OpSymbols} allows to construct $\overline{\op(D)}^{\ii t} \in M(C^*(G))$ through functional calculus for $t\in \R$. One would hope to have some actual symbol $D^{\ii t} \in S^{\ii kt}_p(G)$ that would extend to the multiplier obtained through functional calculus. This is possible (even for the pseudodifferential operators) thanks to estimates on the heat kernels in the filtered calculus obtained in \cite[Thm. 2]{DaveHeat}.

\begin{theorem}[\cite{DaveHeat}]\label{ComplexSymbol}Let $\G \to M$ be a smooth family of graded groups over a compact base. Let $D \in \Gamma(M,\mathcal{U}_k(\Lie(\G)))$  be Rockland, essentially self-adjoint and non-negative with $k> 0$ an even number. There is a continuous family of symbols $D_{\ii t} \in S^{\ii kt}_p(\G)$ such that $\op(D_{\ii t})$ extends to the multiplier $\overline{\op(D)}^{\ii t}$ obtained from $\overline{\op(D)}$ through functional calculus. These symbols are also compatible with the product: $D_{\ii t_1}D_{\ii t_2} = D_{\ii (t_1+t_2)} \mod \CCC^{\infty}_c(\G)$. In particular the corresponding principal symbols form a group.\end{theorem}

We end this section by an analog of Proposition \ref{ContField} for the algebra of principal symbols.

\begin{proposition}Let $\G \to M$ be a bundle of graded Lie groups. The algebra of principal symbols $\Sigma(\G)$ is a continuous field of $C^*$-algebras over $M$ with fiber at $x \in M$ equal to $\Sigma(\G_x)$.\end{proposition}
\begin{proof}
It is clear from the fact that the distributions are fibered over $M$ that $\Sigma(\G)$ is a $\CCC_0(M)$-algebra and the fiber at $x \in M$ is equal to $\Sigma(\G_x)$. The definition of the norm from \ref{OpSymbols} makes $x\mapsto \|\sigma_x\|$ continuous for $\sigma \in \Sigma(\G)$ by Proposition \ref{ContField}.\end{proof}

\begin{remmark}If $\G = G_0 \times M$ is a trivial bundle then $\Sigma(\G) = \Sigma(G_0) \otimes \CCC_0(M)$. \end{remmark}

\section{Construction of the isomorphism for symbols}\label{Section:IsomorphismSymbols}

Let $\G \to M$ denote a bundle of graded Lie groups over a compact base. We want to construct an isomorphism: 
$$\Phi \colon \Sigma(\G)\rtimes \R \to C^*_0(\G).$$
We first need to construct the $\R$ action on $\Sigma(\G)$ underlying the crossed product. For that, let $\Delta_0 \in \Gamma(\mathcal{U}(\Lie(\G)))$ be an essentially self-adjoint, non-negative Rockland symbol of positive even order $m$. We can consider the complex powers \(\Delta_0^{\ii t/m}, t\in \R\). Thanks to Theorem \ref{ComplexSymbol}  they correspond to a continuous family of principal symbols $\Delta_0^{\ii t/m} \in \Sigma^{\ii t}(G)$, $t \in \R$ which is a group under composition law for symbols. Since $\Delta_0$ is self-adjoint they also are unitaries. Let $\sigma \in \Sigma^0_c(G)$ and $t \in \R$, the symbol $\Ad(\Delta_0^{\ii t/m})\sigma := \Delta_0^{\ii t/m}\sigma\Delta_0^{-\ii t/m}$ is of order $0$. We hence obtain an action of $\R$ on $\Sigma^0_c(G)$. This action is continuous because the family of principal symbols obtained from Theorem \ref{ComplexSymbol} is continuous (and the product of symbols is continuous). The action thus extends to the $C^*$-completion $\Sigma(\G)$. We can also see directly the action \(\R \act \m(C^*_0(\G))\) by \(t \mapsto \Ad(\Delta_0^{\ii t/m})\) which is continuous by property of the functional calculus. Theorem \ref{ComplexSymbol} and the composition of symbols then show that this action restricts to the subalgebra \(\Sigma(\G)\).

We now want to define the morphism $\Phi$ sending the crossed product to $C^*_0(\G)$. To do this we need a (strongly continuous) family of unitary multipliers $u_t \in M_u(C^*_0(\G))$ and a morphism $\vphi \colon \Sigma(\G) \to M(C^*_0(\G))$ with the property that 
\begin{equation}\label{crossedprod}
\forall t \in \R, \sigma \in \Sigma(\G), \varphi\left(\Ad(\Delta_0^{\ii t/m})\sigma\right) = u_t \vphi(\sigma) u_t^*.
\end{equation}
These two constructions are similar and rely on a lemma due to Taylor. This lemma's purpose is to represent each class of principal symbols by a particular distribution on $\G$ which will be homogeneous on the nose. To use this lemma however we need to use tempered distributions on the fiber. To do this, recall that the Schwartz class of functions on a vector space does not depend on the choice of basis (nor on the norm). Therefore on a bundle of graded Lie groups, we denote by $\mathscr{S}(\G)$ the space of smooth sections\footnote{If the base \(M\) is not compact we also ask the section to be compactly supported over \(M\).} that are of Schwartz type on the fibers (we identify $\G$ and $\Lie(\G)$ as bundles over $M$). To use the groupoid convolution we extend the definition to half-densities. Denote by $\mathscr{S}(\G,\hd)$ the space of sections which, in any trivialization of $\Lambda^{1/2}\Lie(\G)$, become an element of $\mathscr{S}(\G)$. This definition indeed does not depend on the choice of trivialization (it only changes the Schwartz semi-norms to other equivalent ones). Notice that since for smooth families of groups we have \(r = s\), so \(\hd = r^*\Lambda^1\Lie(\mathcal{G})\) is a bundle of \(1\)-densities.

The space $\mathscr{S}(\G,\hd)$ contains $\CCC^{\infty}_c(\G,\hd)$. Moreover the groupoid convolution extends to a continuous product on $\mathscr{S}(\G,\hd)$, see \cite{HomogeneousGroups}. This makes the Schwartz algebra a subalgebra of $C^*(\G)$ (it is also stable under the involution).

On the space of Schwartz sections, we define a Fourier transform. Identifying $G$ with $\Lie(\G)$, it becomes:
\begin{align*}
\mathcal{F} \colon \mathscr{S}(\G,\hd) &\to \mathscr{S}(\Lie(\G)^*) \\
f &\mapsto \left((x,\eta) \mapsto \int_{\xi \in \g_x} e^{\ii \lag\eta,\xi\rag}f(x,\xi)\right)
\end{align*}
Consider the dual action of $\R^*_+$ on $\Lie(\G)^*$ given by $(\tsp\diff\delta_{\lambda})_{\lambda >0}$. The Fourier transform has the following equivariance property:
$$\forall \lambda > 0, \mathcal{F} \circ \delta_{\lambda}^* = \tsp\diff\delta_{\lambda}^* \circ \mathcal{F}$$

Among the Schwartz sections, denote by $\mathscr{S}_0(\G,\hd)$ those for which the Fourier transform vanishes at infinite order on the zero section of $\Lie(\G)^*$. It is a closed subalgebra of the Schwartz algebra (for the topology induced by the Schwartz semi-norms).

\begin{definition}Let $s \in \CC$ be a complex number, denote by $\mathcal{K}^s(\G)\subset \mathscr{S}_0(\G,\hd)'$ the set of fibered tempered distributions on $G$ that have singular support contained in the unit section and are homogeneous of degree $s$ with respect to the family of dilations $(\delta_{\lambda})_{\lambda > 0}$. \end{definition}

\begin{lemma}[Taylor \cite{Taylor}, Prop. 2.4]\label{Taylor1}Let $k \in \CC$, $u \in S^k_p(\G)$. Then there exists a unique smooth function $v \in \CCC^{\infty}(\Lie(\G)^*\setminus M)$ such that:
\begin{itemize}
\item $v$ is homogeneous of degree $k$: $\forall \lambda > 0, v \circ \tsp\diff\delta_{\lambda} = \lambda^k v$
\item if $\chi\in \CCC^{\infty}_c(\Lie(\G)^*)$ is equal to $1$ on a neighborhood of the zero section then $\hat{u} - (1-\chi)v \in \mathscr{S}(\Lie(\G)^*)$.
\end{itemize}
Conversely if a function $v \in \CCC^{\infty}(\Lie(\G)^*\setminus M)$ is homogeneous of degree $k$ then one can find a symbol $u \in S^k_c(\G)$ such that the second property is satisfied. \end{lemma}

\begin{corollary}\label{QuasiVSExact}There is a linear isomorphim $\Theta \colon \Sigma^s_p(\G) \to \mathcal{K}^s(\G)$ for each $s \in \CC$, which is compatible with the convolution and adjoints. 
%Here $n$ denotes the homogeneous dimension of $\G$. 
\end{corollary}

\begin{proof}
Let $u \in S^s_p(\G)$. Define $v$ as in Taylor's lemma. 
%The condition on the order $\Re(s)>-n$ ensures that $v \in L^1_{loc}(\Lie(\G)^*)$. 
Using its asymptotic behaviour near the zero section, the function $v$ extends to a tempered distribution 
$$v \in \F(\mathscr{S}_0(\G,\hd))'.$$
In other word, it is a continuous linear form on the space of Schwartz functions on \(\Lie(\G)^*\) that vanish at infinite order on the zero section.
The inverse Fourier transform $w$ of \(v\) is in $\mathscr{S}_0(\G,\hd)'$ and is homogeneous of degree $s$, i.e. $w \in \mathcal{K}^s(\G)$. This construction gives a linear map $S^s_c(\G) \to \mathcal{K}^s(\G)$ which is surjective. By the uniqueness of $v$ in Taylor's lemma, the kernel of this map consists of smoothing symbols. It is thus equal to $S^{-\infty}_p(\G)$. The resulting quotient map $\Theta \colon \Sigma^s_p(\G) \to \mathcal{K}^s(\G)$ is then an isomorphism.
\end{proof}

\begin{lemma}[Christ, Geller, Głowacki, Polin \cite{Groupsdilation}, Prop. 2.2] Let $u \in \mathcal{K}^s(\G)$. The convolution by $u$ is a continuous map from $\mathscr{S}_0(\G,\hd)$ to itself.\end{lemma}

\begin{proposition}Let $k \in \mathcal{K}^s(\G)$ be a homogeneous distribution of degree $s \in \CC$. If $\Re(s) = 0$ then $k$ extends to an element of $M(C^*_0(G))$. Moreover 
$$ \|k\| = \sup_{x\in M} \|k_x\|_{L^2(G_x)} = \sup_{x\in M, \pi \in \hat{G_x}}\|\pi(k_x)\|.$$
\end{proposition}

\begin{proof}
    First, replacing \(k\) by \(k^*k\), we may assume that \(s= 0\). Then, using \cite[Thm. 6.19]{HomogeneousGroups}, we obtain that convolution with \(k_x\) is bounded on \(L^2(\G_x)\). This bound is continuous in \(x\) so we get \(\sup_{x\in M} \|k_x\|_{L^2(\G_x)} < +\infty\) by compactness of \(M\). This automatically gives the last equality.
    To obtain that \(k\) extends to a multiplier of \(C^*_0(\G)\) we use the fact that \(C^*_0(\G)\) is the closure of \(\mathscr{S}_0(\G,\hd)\) in the regular reprensentation since \(\G\) is amenable.
\end{proof}

\begin{proposition}\label{MultSymbols}The composition of the map $\Theta$ with the injection of the previous proposition, $\mathcal{K}^0(\G) \hookrightarrow M(C^*_0(\G))$, extends to a *-homomorphism 
\[\varphi \colon \Sigma(\G) \to M(C^*_0(\G)).\]
\end{proposition}
\begin{proof}
It remains to show that this composition is continuous but this is a consequence of the previous proposition since for $u \in S^0_c(\G), x \in M, \pi \in \hat{\G_x}$ we have:
$$\|\pi(\Theta(u)_x)\| \leq \|u\|_{L^2(\G)}.$$
Indeed, $\beta := u-\Theta(u) \in \mathscr{S}(\G,\hd)$ and by the Plancherel formula:
$$\|\pi\circ\delta_{\lambda}(\beta)\| \leq \tr(\pi\circ\delta_{\lambda}(\beta^*\ast\beta))\| = \int_{\tsp\delta_{\lambda}(\mathcal{O}_{\pi})}\widehat{\beta^*\ast\beta}(x,\xi)\diff \xi \xrightarrow[\lambda \to +\infty]{} 0.$$
The Plancherel formula can be found in \cite{CorwinGreenleaf}, here $\mathcal{O}_{\pi}$ denotes the coadjoint orbit in $\Lie(\G_x)^*$ corresponding to the representation $\pi$ through Kirillov's orbit method. Since $\Theta(u)$ is homogeneous of degree $0$ we have $\|\pi(\Theta(u)_x)\| = \|\pi\circ\delta_{\lambda}(\Theta(u)_x)\|$ for every $\lambda > 0$. We thus get $\|\pi(\Theta(u)_x)\| \leq \lim_{\lambda\to +\infty} \|\pi(u_x)\| \leq \|u\|$.
\end{proof}

In the same way, let us denote by $u_t$ the extension as an element of $M(C^*_0(\G))$ of $\Theta(\Delta_0^{\ii t/m})$. Since $\Theta$ preserves the adjoints then $u_t$ is an unitary multiplier and since it preserves the convolution then $u_tu_s = u_{t+s}$ and we get the relation \eqref{crossedprod}. The family \(\Delta_0^{\ii t/m}, t\in\R\) being continuous, the family of unitary multipliers is strongly continuous.
We have thus constructed from $\varphi$ and $(u_t)_{t\in \R}$ a *-homomorphism $\Phi \colon \Sigma(\G)\rtimes \R \to C^*_0(\G)$ by universal property of the crossed product. The rest of this section is devoted to the proof that $\Phi$ is an isomorphism.

To show that $\Phi$ is an isomorphism, we need a better understanding of the representations of $\Sigma(\G)$. This can be done through Pedersen's stratification of \(\G\) defined in Section \ref{Section:Stratification for Family}. It is however more convenient to do it for a single group since the stratification is better known in this case. We thus reduce our study to a single graded Lie group. This follows from the following result:

\begin{proposition}\label{ChpContMorph}The algebras $\Sigma(\G)$ and $C^*_0(\G)$ are bundles of $C^*$-algebras over $M$ with respective fibers at $x \in M$ equal to $\Sigma(\G_x)$ and $C^*_0(\G_x)$. Moreover the action of $\R$ on $\Sigma(\G)$ preserves the fibers and $\Phi$ restricts to a map from $\Sigma(\G_x)\rtimes \R$ to $C^*_0(\G_x)$.
\end{proposition}

It follows that the map \(\Phi \colon \Sigma(\G)\rtimes \R\to C^*_0(\G)\) is an isomorphism if and only if all the induced maps \(\Sigma(\G)\rtimes \R\to C^*_0(\G)\) are also isomorphisms. These maps correspond to the map \(\Phi\) for the groups \(\G_x, x\in M\). It is thus enough to prove the isomorphism for a single graded group. This reduction to a single group makes the proof clearer. Notice however that the same proof would work for a family of graded groups using the results of Section \ref{Section:Stratification for Family}. The other main t

We can now consider a graded Lie group $G$ and prove that $\Phi \colon \Sigma(G)\rtimes \R \to C^*_0(G)$ is an isomorphism. This reduction to a single group makes the proof clearer. Notice however that the same proof would work for a family of graded groups using the results of Section \ref{Section:Stratification for Family}. The other main result used in the proof, Theorem \ref{FFK}, can also be adapted directly to a family of graded groups. The existence of a homogeneous quasi-norm on a family of graded groups follows from a partition of unity argument.

Since $\vphi$ maps $\Sigma(G)$ to $M(C^*_0(G))$, every non-trivial irreducible representation of $G$ induces a representation of $\Sigma(G)$. Moreover, since the image of $\vphi$ consists of invariant elements under the inhomogeneous $\R^*_+$-action then $\pi\circ \vphi = \delta_{\lambda}(\pi)\circ \vphi$ for all $\lambda > 0$ and $\pi \in \widehat{G}\setminus\{1\}$. The following theorem due to Fermanian-Kammerer and Fischer asserts that these representations of $\Sigma(G)$ are irreducible and allow to recover the whole spectrum  of $\Sigma(G)$.

\begin{theorem}[Fermanian-Kammerer, Fischer \cite{FermanianFischer}, Prop. 5.6]\label{FFK}Let $G$ be a graded Lie group. %Choose $| \cdot |$ an homogeneous quasi-norm on $G$, this induces a homeomorphism $\faktor{\widehat{G}\setminus \{1\}}{\R^*_+} \cong \{\chi \in \widehat{G} \ / \ |\chi | = 1\} $.
For $\pi \in \widehat{G}\setminus \{1\}$ denote by $[\pi]$ its class in $\faktor{(\widehat{G}\setminus \{1\})}{\R^*_+}$. For every $\pi \in \widehat{G}\setminus \{1\}$, we have $\pi \circ \vphi \in \widehat{\Sigma(G)}$ and the map 
\begin{align*}
R \colon \faktor{(\widehat{G}\setminus \{1\})}{\R^*_+} &\to \widehat{\Sigma(G)} \\
[\pi] &\mapsto \pi\circ \vphi
\end{align*}
is a homeomorphism.
\end{theorem}
\begin{proof}Since our definition of the symbol algebra is different than the one of \cite{FermanianFischer} we detail the proof in our context.
Let us prove that $R$ is injective. Let $\pi_1,\pi_2 \in \widehat{G}\setminus\{1\}$ be non trivial unitary irreducible representations of $G$ with $[\pi_1]\neq [\pi_2]$. Since $\faktor{(\widehat{G}\setminus \{1\})}{\R^*_+}$ is a $T_0$ space (it is the spectrum of the algebra $C^*_0(G)\rtimes \R^*_+$) then either $[\pi_1] \notin \overline{\{[\pi_2]\}}$ or $[\pi_2] \notin \overline{\{[\pi_1]\}}$\footnote{This can also be seen by using Pedersen's stratification.}. Without loss of generality we will assume the former. We can use a similar reasoning than the one in \cite{Hebisch} to construct a function $f \in \mathscr{S}_0(G)$ such that for every $\pi \in \overline{\R^*_+\cdot\pi_2}$ we get $\pi(f) = 0$ and for every other unitary irreducible representation $\pi$ we have $\pi(f)$ positive and left invertible. We can recover a symbol from $f$ with $\sigma = \int_{0}^{+\infty}\delta_{\lambda *}f \frac{\diff \lambda}{\lambda}$. Indeed if we take its Fourier transform, we get:
\[\int_{0}^{+\infty} \tsp\diff\delta_{\lambda*}\mathcal{F}(f)\frac{\diff \lambda}{\lambda}.\]
This integral converges absolutely since \(\mathcal{F}(f)\) decays rapidly at infinity and near \(0\). This integral also defines a homogeneous function with respect to the inhomogeneous dilations on \(\g^*\). We can then apply the same reasoning as in the proof of Corollary \ref{QuasiVSExact}. We thus get \(\sigma \in \mathcal{K}^0(G)\) which we identified to \(\Sigma^0_p(G)\).

If $\pi \in \widehat{G}\setminus\{1\}$ we have:
$$\pi\circ\vphi(\sigma) = \int_0^{+\infty}(\lambda\cdot \pi) (f) \frac{\diff \lambda}{\lambda}.$$
Therefore we have $R([\pi_2])(\sigma) = 0$ and $R(\pi_1)(\sigma) \neq 0$ and $R(\pi_1) \neq R(\pi_2)$ (this also shows that the representations obtained through $R$ are all nonzero).

We now prove the surjectivity of $R$. We proceed as in \cite{FermanianFischer} and construct a right inverse. Let $|\cdot |$ be a homogeneous quasi-norm on $G$. Take $f\in \CCC^{\infty}_c(G,\hd)$ and take its (euclidean) Fourier transform $\hat{f}$. Now the function $\tsp\diff\delta_{|\cdot|^{-1}}^*\hat{f}\in \CCC^{\infty}(\g^*\setminus{0})$ is homogeneous of degree 0. 
By virtue of Lemma \ref{Taylor1}, it corresponds to a unique symbol $\sigma_f \in \Sigma^0(G)$. Let $\rho$ be a representation of $\Sigma(G)$, $\pi_{\rho} \colon f\to \rho(\sigma_f)$ extends to a continuous representation of $C^*(G)$. The formula for generalized character of unitary representations of nilpotent Lie groups (see \cite{Kirillov,CorwinGreenleaf}) shows that $\pi_{\rho}$ is the same representation as the one constructed in \cite{FermanianFischer} Lemma 5.7. 
Therefore we have that $\pi_{\rho}$ is a non-trivial unitary irreducible representation of $G$ and the map $\rho \mapsto [\pi_{\rho}]$ is a right inverse to $R$.
Thus far, we have proven that $R$ was a continuous bijection. Its inverse being the map $\rho \mapsto [\pi_{\rho}]$ which is also continuous, $R$ is therefore a homeomorphism.
\end{proof}

\begin{corollary}\label{Corollary:SpectrumSymbolsFamily}
Let \(\G \to M\) be a smooth family of graded groups then the spectrum of $\Sigma(\G)$ is homeomorphic to $\faktor{(\widehat{\G}\setminus M)}{\R^*_+}$ and thus to the double quotient \(\faktor{\left(\faktor{\Lie(\G)^*\setminus M}{\Ad^*(\G)}\right)}{\R^*_+}\).
\end{corollary}
\begin{proof}
Thanks to Proposition \ref{ChpContMorph} the result reduces to each fiber over the points of $M$, we can then apply Theorem \ref{FFK}.
\end{proof}

We are now ready to prove the main theorem of this section.

\begin{theorem}\label{ProofOfIsom}
Let $G$ be a graded Lie group, then 
$$\Phi \colon \Sigma(G)\rtimes \R \to C^*_0(G),$$
is an isomorphism of $C^*$-algebras.
\end{theorem}

\begin{corollary}Let $\G$ be a smooth family of graded Lie groups, then 
$$\Phi \colon \Sigma(\G)\rtimes \R \to C^*_0(\G),$$
is an isomorphism of $C^*$-algebras.
\end{corollary}

\begin{proof}[Proof of Theorem \ref{ProofOfIsom}]
Let $\empty = V_0 \subset V_1 \subset \cdots \subset V_d = \widehat{G}\setminus \{1\}$ be a fine stratification of $\widehat{G}$ as in Theorem \ref{Stratification}. Recall that every $V_k$ is open and $\R^*_+$-invariant, that the spaces $\Lambda_k := V_k \setminus V_{k-1}$ are Hausdorff and that the action $\R^*_+ \act \Lambda_k$ is free and proper.
This stratification induces a filtration of $\Sigma(G)$ and $C^*_0(G)$ into sequences of increasing ideals:
\begin{align*}
    \{0\} = J_0 \triangleleft J_1 \triangleleft \cdots \triangleleft J_d = C^*_0(G) \\
    \{0\} = \Sigma_0 \triangleleft \Sigma_1 \triangleleft \cdots \triangleleft \Sigma_d = \Sigma(G) 
\end{align*}
with $J_k = \bigcap_{\pi \in \widehat{G}\setminus V_k}\ker(\pi)$ and $\Sigma_k = \bigcap_{\pi \in \widehat{G}\setminus V_k}\ker(\pi\circ \vphi)$. By construction the subsets $\Sigma_k$ are $\R$-invariant and $\Phi$ restricts to maps $\Phi \colon \Sigma_k \rtimes \R \to J_k$.
The spectrum of $\faktor{J_k}{J_{k-1}}$ is $\widehat{\faktor{J_k}{J_{k-1}}} = \widehat{J_k}\setminus \widehat{J_{k-1}} = V_k \setminus V_{k-1} = \Lambda_k$ for every $k$. Recall that the quotient algebra $\faktor{J_k}{J_{k-1}}$ is a continuous field of $C^*$-algebras over its spectrum $\Lambda_k$ and by Corollary \ref{SubquotientsCompact} (see also \cite{CStarNilp}) the fiber at $\lambda \in \Lambda_k$ is $\mathcal{K}(\mathcal{H}_{k})$ where $\mathcal{H}_k$ is a Hilbert space (of dimension $1$ for $k=d$ and of infinite dimension otherwise). Since the action of $\R^*_+$ on $\Lambda_k$ is free and proper, $\faktor{\Lambda_k}{\R^*_+}$ is also a Hausdorff space and each $\faktor{\Sigma_k}{\Sigma_{k-1}}$ is a continuous field of $C^*$-algebras over $\faktor{\Lambda_k}{\R^*_+}$. The fiber at $[\lambda] \in \faktor{\Lambda_k}{\R^*_+}$ is the image of $\lambda \circ \vphi$. It is a sub-algebra of $\mathcal{B}(\mathcal{H}_{k})$, we will denote it by $A_{\lambda}$.

\begin{lemma}
For every $\lambda \in \widehat{G} \setminus \{1\}$, $A_{\lambda}$ is a simple algebra.
\end{lemma}
\begin{proof}
The spectrum of a continuous field of $C^*$-algebra is in bijection with the disjoint union of the spectra of the fibers. Here the spectrum of the continuous field is exactly its base hence the fibers have trivial spectrum and are thus simple. Indeed having a trivial spectrum means having no non-trivial primitive ideal but since every ideal is an intersection of primitive ideals then an algebra with no non-trivial primitive ideals is simple.
\end{proof}

\begin{lemma}
For every $\lambda \in \widehat{G} \setminus \{1\}$, $A_{\lambda}$ contains $\mathcal{K}(\mathcal{H}_{k})$ (where $\lambda \in \Lambda_k$).
\end{lemma}

\begin{proof}
As in the proof of Theorem \ref{FFK}, let us choose a homogeneous quasi-norm on $G$. This induces a $*$-homomorphism $\mathscr{S}_0(G,\hd) \to \Sigma(G)$ and if $\lambda \in \widehat{G}$ is of norm 1 then the previous morphism composed with $\lambda$ seen as a representation of $\Sigma(G)$ extends to $\lambda$ seen as a representation of $C^*_0(G)$. Thus $A_{\lambda}$ contains $\lambda(C^*(G)) = \mathcal{K}(\mathcal{H}_{k})$.
\end{proof}

\begin{corollary}The algebras $\faktor{\Sigma_k}{\Sigma_{k-1}}$ are trivial fields of $C^*$-algebras:
$$\faktor{\Sigma_k}{\Sigma_{k-1}} \cong \CCC_0\left(\faktor{\Lambda_k}{\R^*_+},\mathcal{K}(\mathcal{H}_k)\right).$$
\end{corollary}

\begin{corollary}
For every $k \geq 0$, the action of $\R$ on $\faktor{\Sigma_k}{\Sigma_{k-1}}$ is inner and $\Phi \colon \faktor{\Sigma_k}{\Sigma_{k-1}} \rtimes \R \to \faktor{J_k}{J_{k-1}}$ is an isomorphism.
\end{corollary}
\begin{proof}
%Let us fix a number $t\in \R$. Since $\Delta_0^{\ii t/m}$ acts as a multiplier $u_t$ of $\faktor{J_k}{J_{k-1}}$ then for each representation $\lambda \in \Lambda_i$, $\lambda(u_t)\in \mathcal{B}(\mathcal{H}_i)$. The operator $\Delta_0^{\ii t/m}$ thus preserves each fiber of the continuous field of $C^*$ algebras defining $\faktor{\Sigma_i}{\Sigma_{i-1}}$. We then need to show that its norm is bounded but being a multiplier of $\faktor{J_i}{J_{i-1}}$ we have:
%$$\sup_{\lambda \in \Lambda_i} \|\lambda(\Delta_0^{\ii t/m})\|_{\mathcal{B}(\mathcal{H}_{\lambda})} < +\infty$$
%Therefore we have:
%$$\Delta_0^{\ii t/m} \in \CCC_b(\faktor{\Lambda_i}{\R^*_+},\mathcal{K}(\mathcal{H}_i)) \subset \m\left(\faktor{\Sigma_i}{\Sigma_{i-1}}\right).$$
%The crossed product $\faktor{\Sigma_i}{\Sigma_{i-1}} \rtimes \R$ is therefore trivial and we have:
%$$\faktor{\Sigma_i}{\Sigma_{i-1}} \rtimes \R \cong \CCC_0(\Lambda_i,\mathcal{K}(\mathcal{H}_i)).$$
%The identification $\faktor{\Lambda_i}{\R^*_+} \times \R^*_+ \cong \Lambda_i$ is made by the choice of quasi-norm on $G$.
We fix a quasi-norm on \(G\), this induces an identification \(\faktor{\Lambda_k}{\R^*_+} \times \R^*_+ \cong \Lambda_k\). Consider the function \(\chi_k \colon \Lambda_k \to \CC\) which through this identification, reads \(\chi_k(\theta,r) = r^m\). We can see \(\chi_k\) as a central unbounded multiplier of \(\faktor{J_k}{J_{k-1}} \cong \CCC_0(\Lambda_k,\mathcal{K}(\mathcal{H}_k))\).

Let us fix \(t\in \R\), we know that \(\Delta_0^{\ii t/m}\) restricts to a multiplier of \(\faktor{J_k}{J_{k-1}}\). We can see \(\Delta_0^{\ii t/m} \in \CCC_b(\Lambda_k,\mathcal{K}(\mathcal{H}_k))\). This function is homogeneous of degree \(\ii t\) so not invariant. 
Consider thus \(\left(\chi_k^{-1}\Delta_0\right)^{\ii t/m} \in \CCC_b(\Lambda_k,\mathcal{K}(\mathcal{H}_k))\), it is a \(\R^*_+\)-invariant function so we can see it as a function on the quotient \(\left(\chi_k^{-1}\Delta_0\right)^{\ii t/m} \in \CCC_b(\faktor{\Lambda_k}{\R^*_+},\mathcal{K}(\mathcal{H}_k))\), i.e. \(\left(\chi_k^{-1}\Delta_0\right)^{\ii t/m} \in \m\left(\faktor{\Sigma_k}{\Sigma_{k-1}}\right)\).

Since \(\chi_k\) is central, we get \(\Ad\left(\Delta_0^{\ii t/m}\right) = \Ad\left(\left(\chi_k^{-1}\Delta_0\right)^{\ii t/m}\right)\) and therefore the action on the subquotient is inner. Because of this the crossed product becomes a tensor product and we get:
\begin{align*}
    \faktor{\Sigma_k}{\Sigma_{k-1}} \rtimes \R &\cong \CCC_b(\faktor{\Lambda_k}{\R^*_+},\mathcal{K}(\mathcal{H}_k)) \otimes \CCC_0(\R_+^*) \\
    &\cong \CCC_0(\faktor{\Lambda_k}{\R^*_+} \times \R^*_+,\mathcal{K}(\mathcal{H}_k)) \\
    &\cong \CCC_0(\Lambda_k,\mathcal{K}(\mathcal{H}_k)) \\
    &\cong \faktor{J_k}{J_{k-1}}.\qedhere
\end{align*}
\end{proof}
We have thus far showed that the maps:
$$\Phi \colon \faktor{\Sigma_k}{\Sigma_{k-1}} \rtimes \R \to \faktor{J_k}{J_{k-1}},$$
were isomorphisms of $C^*$-algebras. We can then use the respective exact sequences to show inductively that each
$$\Phi \colon \Sigma_k \rtimes \R \to J_k,$$
is an isomorphism (starting from $k = 0$). The result for $k = d$ then concludes the proof.
\end{proof}

As a corollary of the proof we obtain a decomposition result for the symbol algebra:

\begin{theorem}Let $G$ be a graded Lie group, $\emptyset = V_0 \subset V_1 \subset \cdots \subset V_d = \widehat{G}\setminus \{1\}$ a Pedersen stratification of its unitary dual. Denote by $\Sigma_d(G)$ the ideal corresponding to the subspace $\faktor{V_d}{\R^*_+}$ of the spectrum $\widehat{\Sigma(G)}$. This gives a nested sequence of ideals:
$$\{0\} = \Sigma_0(G) \triangleleft \Sigma_1(G) \triangleleft \cdots \triangleleft \Sigma_d(G) = \Sigma(G)$$
and we have isomorphisms:
$$\faktor{\Sigma_k(G)}{\Sigma_{k-1}(G)} \cong \CCC\left(\faktor{\Lambda_k}{\R^*_+},\mathcal{K}_k\right).$$
Here $\Lambda_k = V_k\setminus V_{k-1}$, $\faktor{\Lambda_k}{\R^*_+}$ is a locally compact Hausdorff space and $\mathcal{K}_k$ is the algebra of compact operators on a separable Hilbert space (of dimension \(1\) if $k = d$ and infinite dimensional otherwise).
\end{theorem}

This result, along with Corollary \ref{Corollary:SpectrumSymbolsFamily}, Proposition \ref{ChpContMorph} and the extension of Pedersen stratification given in Theorem \ref{Stratification General family} allow us to decompose the algebra of principal symbols in the general case.

\begin{theorem}\label{Thm:GeneralizationEpsteinMelrose}
    Let \(\G \to M\) be a smooth family of graded groups. Consider a Pedersen stratification \(\emptyset = \mathcal{V}_0 \subset \mathcal{V}_1 \subset \cdots \subset \mathcal{V}_r = \widehat{\mathcal{G}}\setminus M\) and let \(\Omega_k = \mathcal{V}_k\setminus\mathcal{V}_{k-1}\) denote the different strata.

    Let \(\Sigma_{k}(\G)\) be the ideal of \(\Sigma(\G)\) corresponding to the subspace \(\faktor{\mathcal{V}_k}{\R^*_+}\) of the spectrum. This gives a nested sequence of ideals:
    \[\{0\} = \Sigma_0(\G) \triangleleft \Sigma_1(\G) \triangleleft \cdots \triangleleft \Sigma_d(\G) = \Sigma(\G)\]
    and we have isomorphisms:
    \[\faktor{\Sigma_k(\G)}{\Sigma_{k-1}(\G)} \cong \CCC_0\left(\faktor{\Omega_k}{\R^*_+},\mathcal{K}_k\right).\]
\end{theorem}

This is a generalization of Epstein and Melrose decomposition of the symbol algebra in the contact case \cite{EpsteinMelrose}. They showed that if $(M,H)$ is a (compact) contact manifold then there is an exact sequence:
$$\xymatrix{0 \ar[r] & \CCC(M,\mathcal{K}\oplus \mathcal{K}) \ar[r] & \Sigma(T_HM) \ar[r] & \CCC(\mathbb{S}H^*)\ar[r] & 0.}$$
Regarding our result this corresponds to the Pedersen stratification of the Heisenberg group associated to $H_x$ for each $x \in M$. Recall that this consists of $\Lambda_1 = \R^*$ and $\Lambda_2 = H_x^*$ and the $\R^*_+$-action on each component is a regular dilation.

\begin{example}The examples given by the abelian case and the Heisenberg groups might be misleading on one point. The strata $\faktor{\Lambda_k}{\R^*_+}$ might not be compact, as shown in the following example.Consider the Engel group whose Lie algebra $\g$ is generated by $X,Y,Z,T$ with $[X,Y] = Z, [X,Z] = T$ and the other brackets being zero. This is a 3-step nilpotent Lie group. Its representation theory can be computed through Kirillov's theory \cite{Kirillov} and its Pedersen stratification gives for the symbol algebra:
$$\faktor{\Lambda_1}{\R^*_+} \cong \R \sqcup \R ; \faktor{\Lambda_2}{\R^*_+} \cong \{pt\} \sqcup \{pt\} ; \faktor{\Lambda_3}{\R^*_+} \cong \mathbb{S}^1.$$
The first strata is therefore non-compact.
\end{example}

\begin{remmark}
    Throughout this section, we have used that \(M\) was compact to use some uniform boundedness arguments. Even though the isomorphism happens over each fiber of \(M\), it is not clear if the operator \(\Delta_0\) could be chosen with a good behaviour at infinity in \(M\) so that its complex powers could be defined (they can be defined locally but we have no way to ensure that they stay bounded at infinity in \(M\)). This proof could be adapted if \(M\) had a reasonable enough compactification (e.g. with a boundary or corners) but I don't know if it would hold in the case of any non-compact manifold.

    This is however not a problem in the decomposition of the symbol algebra. Indeed the identifications of the subquotients can be done locally so Theorem \ref{Thm:GeneralizationEpsteinMelrose} also holds when \(M\) is non-compact.
\end{remmark}

\section{Calculus on filtered manifolds}

Let us now go back to the case of a filtered manifold $(M,H)$. There is an analog of Connes' tangent groupoid \cite{connes1994noncommutative} denoted by $\T_HM$. Algebraically it has the form:
$$M \times M \times \R^* \sqcup T_HM \times \{0\} \rr M \times \{0\}.$$
Its algebroid is given by the following algebra of sections:
$$\Gamma(\Lie(\T_HM))  = \{X \in \mathfrak{X}(M \times \R) \ / \ \forall k \geq 0, \partial_t^k X_{|t = 0} \in \Gamma(H^k)\}$$
which is almost injective hence integrable by a theorem of Debord \cite{debord2000groupoides}. The smooth structure is detailed in \cite{van_erp_2017,choi_2019,mohsen2018deformation}. Recall that $\T_HM$ is endowed with the zoom action $\alpha$ of $\R^*_+$ :
\begin{align*}
\alpha_{\lambda}(x,y,t) &= (x,y,\lambda^{-1}t)\\
\alpha_{\lambda}(x,\xi,0) &= (x,\delta_{\lambda}(\xi),0)
\end{align*}
where $\delta$ is the inhomogeneous action of $\R^*_+$ on $T_HM$ used in the previous section. We briefly recall the definition of pseudodifferential operators in this context, following \cite{van_erpyunken}:

\begin{definition}A pseudodifferential operator in the filtered calculus on $M$ of order $m\in \CC$ is a properly supported, \(r\)-fibered distribution $P \in \mathcal{D}'(M\times M, \hd)$ such that there exists $\PP \in \mathcal{D}'(\T_HM,\hd)$ with the following properties:
\begin{itemize}
\item $\PP$ is properly supported and \(r\)-fibered.
\item $\PP_1 = P$ where $\PP_t = \ev_{t*}\PP$ and $\ev_t$ is the evaluation map on the fiber at time $t \in \R$ on $\T_HM$.
\item $\PP$ is quasi-homogeneous w.r.t the zoom action i.e. 
$$\forall \lambda > 0, \alpha_{\lambda *} \PP - \lambda^m \PP \in \CCC^{\infty}_p(\T_HM,\hd).$$
\end{itemize}
We denote by $\Psi^m_H(M)$ the set of such distributions and $\bbpsi^m_H(M)$ the set of their extensions $\PP$ to $\T_HM$.
\end{definition}

In this definition we assumed the distribution to be fibered with respect to the range map \(r\) of the tangent groupoid. It is proved in \cite{van_erpyunken} that these distributions are also automatically fibered for the source map.

From this setup one can derive the usual properties of a pseudodifferential calculus, namely the algebra structure, the existence of a symbol map, ellipticity criterion, parametrices... The symbol of an operator $P\in \Psi^m_H(M)$ is here obtained by extending it to some quasi-homogeneous $\PP$ and considering $\PP_0 \in S^m_c(T_HM)$. This will depend on the choice of $\PP$ but the class $[\PP_0] \in \Sigma^m_c(T_HM)$ only depends on $P$. We get the exact sequence:
$$\xymatrix{0 \ar[r] & \Psi^{m-1}_H(M) \ar[r] & \Psi^m_H(M) \ar[r] & \Sigma^m_c(T_HM) \ar[r] & 0}.$$
As for the symbols, operators of non-positive order extend to bounded operators on $L^2(M)$ and negative order operators to compact operators. Let us denote by $\Psi^*_H(M)$ the closure of $\Psi^0_H(M)$ in $\mathcal{B}(L^2(M))$. The previous exact sequence for $m = 0$ extends to the following one:
$$\xymatrix{0 \ar[r] & \mathcal{K}(L^2(M)) \ar[r] & \Psi^*_H(M) \ar[r] & \Sigma(T_HM) \ar[r] & 0}.$$
The usual ellipticity condition is here replaced by the Rockland condition on symbols, we will say that a pseudodifferential operator satisfies the Rockland condition if its symbol does. Van Erp and Yuncken proved in \cite{van_erpyunken} that this condition gives the existence of a parametrix for the operator in this calculus: if $P \in \Psi^m_H(M)$ is Rockland if and only if there exists $Q \in \Psi^{-m}_H(M)$ such that $PQ-1,QP-1 \in \CCC^{\infty}(M\times M,\hd)$.
The last thing we will need is the existence of complex powers for some positive operators:

\begin{theorem}[Dave, Haller \cite{DaveHeat}, Thm. 2]\label{ComplexOp} Let $M$ be a compact filtered manifold. Let $P$ be a differential operator of positive even order $d$ which satisfies the Rockland condition and is positive. Then the complex powers $P^z, z \in \CC,$ obtained through functional calculus are pseudodifferential operators of respective order $dz$. Moreover these operators form a holomorphic family of pseudodifferential operators.
\end{theorem}

\section{The Schwartz algebra}

We now define the Schwartz algebra $\mathscr{S}(\T_HM,\hd)$ of the tangent groupoid. Our approach is analogous to the one of Carillo-Rouse in the unfiltered case \cite{CRSchwartz}. A similar approach has been used by Ewert in \cite{Ewert} but she used the coordinates of Choi and Ponge \cite{choi_2019} while we will use the ones of Van Erp and Yuncken \cite{van_erp_2017} (from which we take some terminology for the constructions in this section). Both approaches yield the same algebra. The idea is to define the algebra on exponential charts and show that these algebras can be glued on the whole tangent groupoid. This Schwartz algebra will essentially contain two kinds of functions. The first are the smooth functions on the open subgroupoid $\T_HM_{|\R^*}$ which are compactly supported at each $t \neq 0$ (on a compact subset of \(M\times M\) independent of \(t\)) and have a Schwartz decay at infinity and $0$. We denote this space by:
\[\mathscr{S}(\R^*,\CCC^{\infty}_c(M\times M)).\]
The Schwartz decay condition is then relative to all the semi-norms that define the Fréchet structure on \(\CCC^{\infty}_c(M\times M)\).

The other kind of functions are the Schwartz type functions on exponential charts (in particular their restriction on $T_HM$ will be Schwartz). In order to describe these last functions, we fix $\nabla$ a graded connection on $\gt_HM$, compatible with the dilations and a splitting $\psi \colon \gt_HM \to TM$\footnote{This means that for every $j$, $\psi\colon \faktor{H^j}{H^{j-1}}\to H^j$ is a section of the quotient map.}. Recall that this induces a vector bundle isomorphism 
\[\Psi \colon \gt_HM\times \R \to \Lie(\T_HM).\] 
We fix $\mathcal{U} \subset \gt_HM$ a domain of injectivity, this means that the composition of \(\psi\) with the groupoid exponential 
\[\exp^{\psi\circ \nabla \circ \psi^{-1}} \circ \psi \colon \mathcal{U} \to M \times M,\]
is well defined and injective.
This gives us an open subset 
\[\mathbb{U} := \left\lbrace (x,\xi,t) \in \gt_HM \times \R, (x,\delta_t(\xi))\in \mathcal{U}\right\rbrace\]
which contains both \(\gt_HM\times \{0\}\) and the whole zero section of \(\gt_HM \times \R\).
On this set we can define extend the previous exponential map to an exponential map for the groupoid \(\T_HM\). We define it as:
\begin{align*}
\exp^{\nabla,\psi}\circ \Psi \colon &(x,\xi,t) \mapsto \left(\exp^{\psi \circ \nabla \circ \psi^{-1}}(x,\psi(\delta_t(\xi)),t\right) \\
						            &(x,\xi,0) \mapsto (x,\xi,0).
\end{align*}
Then $\exp^{\nabla,\psi} \circ \Psi \colon \mathbb{U} \to \T_HM$ is a well defined diffeomorphism onto its image $\mathbb{V}$. We have $\mathbb{V} = (T_HM \times \{0\}) \sqcup (\exp^{\psi\circ\nabla\circ\psi^{-1}}(\mathcal{U}) \times \R^*)$ so $\mathbb{V}$ is both a neighborhood of the unit section of $\T_HM$ and of the zero-fiber $T_HM$. We thus have that $\T_HM = \mathbb{V} \bigcup (M\times M \times \R^*)$. We define the Schwartz space on $\mathbb{V}$ by defining one on $\mathbb{U}$ and pushing it forward with the exponential map.

\begin{definition} A function $f \in \CCC^{\infty}(\mathbb{U})$ is in $\mathscr{S}(\mathbb{U})$ if there is a compact set $K \subset \mathcal{U}$ and $T > 0$ such that if $(x,\delta_t\xi,t) \notin K \times [-T,T]$ then $f(x,\xi,t) = 0$ and also $f$ is Schwartz as a function on the bundle of nilpotent Lie algebras  $\gt_HM \times \R \to M \times \R$.

The space $\mathscr{S}(\mathbb{V})$ is defined by the push-forward of $\mathscr{S}(\mathbb{U})$ by the diffeomorphism \(\exp^{\nabla,\psi}\circ\Psi\).

A function $f \in \CCC^{\infty}(\T_HM)$ is a Schwartz type function if there exists an exponential chart $\exp^{\nabla,\psi}\circ \Psi \colon \mathbb{U} \to \mathbb{V}$ such that 
$$f \in \mathscr{S}(\mathbb{V}) + \mathscr{S}(\R^*,\CCC^{\infty}_c(M\times M)).$$
We denote by $\mathscr{S}(\T_HM)$ the space of such functions.
\end{definition}

Remember that $\delta_0 = 0$ so in particular the whole $\gt_HM$, seen as the zero fiber of $\gt_HM \times \R$, can be included in the support the functions in $\mathscr{S}(\mathbb{U})$. For $t \neq 0$ however the support of $f_t = f(\cdot,\cdot,t)$ is contained in $\delta_t^{-1}(K)$ which is compact. 
In particular a smooth function on $\mathbb{U}$ is Schwartz if and only if it has the aforementioned support condition and is Schwartz at $t = 0$ (asking it on $\gt_HM \times \R$ is redundant).

This also means that if we now consider the push forward of \(f\) on \(\mathbb{V} \subset \T_HM\), its support is included in 
\[(T_HM\times\{0\})\sqcup \exp^{\psi\circ\nabla\circ\psi^{-1}}(K)\times ([-T;T]\setminus \{0\}).\] 
Therefore the support at time \(t \neq 0\) of a function in \(\mathscr{S}(\mathbb{V})\) is compact in \(M\times M\), (and goes to \(0\) at infinity since it vanishes after time \(T\)).

\begin{proposition}Let $f = (f_t)_{t \in \R} \in \mathscr{S}(\T_HM)$ then:
\begin{itemize}
\item $f_0 \in \mathscr{S}(T_HM)$
\item for every $t \neq 0$, $f_t \in \CCC^{\infty}_c(M \times M)$ moreover the support of each $f_t$ is contained in a compact set of $M \times M$ independent of $t$
\item the function $t\mapsto f_t$ has rapid decay when $t$ goes to $\pm \infty$.
\end{itemize}
\end{proposition}
\begin{proof}
For the first point we have $\Psi_{|t = 0} = \Id_{\gt_HM}$ so we recover the definition of the Schwartz class on a graded Lie group bundle. The other points result from the discussion above.
\end{proof}

\begin{theorem}
%A smooth function \(f\in \CCC^{\infty}(\T_HM)\) belongs to the Schwartz class if and only if it satisfies %the following conditions:
%\begin{itemize}
%\item $f_0 \in \mathscr{S}(T_HM)$
%\item for every $t \neq 0$, $f_t \in \CCC^{\infty}_c(M \times M)$ moreover the support of each $f_t$ is contained in a %compact set of $M \times M$ independent of $t$
%\item the function $t\mapsto f_t$ has rapid decay when $t$ goes to $\pm \infty$.
%\end{itemize}
The definition of $\mathscr{S}(\T_HM)$ does not depend on the choice of connection, splitting and exponential chart. 
That means that for any connection $\nabla$, splitting $\psi$, domain of injectivity $\mathcal{U}$ and corresponding exponential chart $\mathbb{U} \isomto \mathbb{V} \subset \T_HM$ then:
$$\mathscr{S}(\T_HM) = \mathscr{S}(\mathbb{V}) + \mathscr{S}(\R^*,\CCC^{\infty}_c(M\times M)).$$
\end{theorem}
\begin{proof}
Consider another choice of connection, splitting, domain of injectivity and exponential charts written with a prime instead. Up to considering \(\mathbb{V} \cap \mathbb{V}' \) instead of \(\mathbb{V}\) we may assume that $\mathbb{V}' \subset \mathbb{V}$. We can use a function $\vphi$ compactly supported in $\mathbb{V}$ with value $1$ near the zero fiber so that if $f \in  \mathscr{S}(\mathbb{V})$ then $f = \vphi f + (1-\vphi) f$. Since $(1-\vphi) f \in \mathscr{S}(\R^*,\CCC^{\infty}_c(M\times M))$ we can just show that $\vphi f \in \mathscr{S}(\mathbb{V}')$. Therefore we can restrict ourselves to the case where $\mathbb{V}$ and $\mathbb{V}'$ are defined from the same domain of injectivity $\mathcal{U} \subset \gt_HM$ and show that if $\Psi,\Psi'\colon \gt_HM\times \R \isomto \ttt_HM$ are the respective isomorphisms used to construct the exponential charts then:
$$(\Psi^{-1}\circ \Psi')^*(\mathscr{S}(\mathbb{U})) = \mathscr{S}(\mathbb{U}').$$
Notice that $\Psi^{-1}\circ \Psi' = (\delta_t^{-1} \circ \psi^{-1} \circ \psi' \circ \delta_t)_{t \in \R}$, the value at $t = 0$ is $\Id$. It is thus a vector bundle isomorphism of $\gt_HM \times \R$ and therefore preserves the Schwartz class. We now need to show that the condition on the support is preserved under the transformation given by $\Psi^{-1}\circ \Psi'$. For $K \subset \mathcal{U}$ and $T >0$, we define:
$$K_T := \{(x,\xi,t) \in \gt_HM \times \R , t \in [-T;T] \& (x,\delta_t(\xi)) \in K\}.$$
By definition a function in $\mathscr{S}(\mathbb{U})$ has a support contained in a set of the form $K_T$ for some $T > 0 $ and $K \subset \mathcal{U}$ compact. The computation of $\Psi^{-1}\circ\Psi'$ shows directly that:
$$\Psi^{-1}\circ\Psi' (K_T) = (\Psi^{-1}\circ\Psi'(K))_T,$$
hence the condition on the support is preserved by the transformation $\Psi^{-1}\circ\Psi'$.
\[f \in \mathscr{S}(\mathbb{V}) + \mathscr{S}(\R^*,\CCC^{\infty}_c(M\times M)) = \mathscr{S}(\T_HM).\qedhere\]
\end{proof}

In order to deal with convolution we define a Schwartz class for half-densities denoted by $\mathscr{S}(\T_HM,\hd)$. They are the sections of the half-density bundle for which in one (hence any) zoom-invariant trivialization of the half-density bundle\footnote{This means that taking the pullback by $\alpha_{\lambda}$ scales the measure by a factor $\lambda^{-d_H}$ where $d_H$ is the homogeneous dimension of $M$.} they become Schwartz functions as defined before. To see that this definition is not ambiguous, let us consider two different zoom-invariant trivializations. They differ by a multiplication by a positive smooth function. By zoom invariance this function is zoom-invariant. In particular it is bounded and thus preserves the growth conditions at infinity used to define the Schwartz algebra. With similar notations as before, we now have:
\[\mathscr{S}(\T_HM,\hd) = \mathscr{S}(\mathbb{V},\hd) + \mathscr{S}(\R^*,\CCC^{\infty}_c(M\times M,\hd)).\]

Now, because of the Schwartz decay conditions at infinity and near \(0\), we have \(\mathscr{S}(\R^*,\CCC^{\infty}_c(M\times M,\hd)) \subset L^1(\T_HM,\hd)\). Similarly, by the Schwartz decay conditions on \(\mathbb{U}\), we have \(\mathscr{S}(\mathbb{V},\hd)\subset L^1(\T_HM,\hd)\). Thus \(\mathscr{S}(\T_HM, \hd) \subset L^1(\T_HM,\hd)\), in particular we can consider the groupoid convolution on between elements of \(\mathscr{S}(\T_HM,\hd)\). 

\begin{theorem}The groupoid convolution and adjoint extend to the Schwartz space $\mathscr{S}(\T_HM,\hd)$ making it a $*$-algebra. Moreover we have the inclusions:
$$\CCC^{\infty}_c(\T_HM,\hd) \subset \mathscr{S}(\T_HM,\hd) \subset C^*(\T_HM).$$
\end{theorem}
\begin{proof}
It remains to show that the result of the convolution is still in the Schwartz class. Let us fix a splitting $\psi$, an equivariant connection $\nabla$, consider a domain of injectivity $\mathcal{U} \subset \gt_HM$ and let $\mathbb{U} \subset \gt_HM \times \R$ the corresponding open subset and $\exp^{\nabla,\psi} \colon \mathbb{U} \isomto \mathbb{V}\subset \T_HM$ the resulting exponential chart. Since the convolution will force us to consider products in $\T_HM$ we choose an open subset $\mathcal{U}' \subset \mathcal{U}$ such that $m_{\T_HM}(\mathbb{V}' \ {}_s\! \! \times_r  \mathbb{V}') \subset \mathbb{V}$. Let $f,g \in \mathscr{S}(\T_HM,\hd)$, we decompose them as $f = f_1 + f_2, g = g_1 + g_2$ with $f_1,g_1 \in \mathscr{S}(\mathbb{V}',\hd)$ and $f_2,g_2 \in \mathscr{S}(\R^*,\CCC^{\infty}_c(M\times M,\hd))$. We have:
$$f\ast g = f_1\ast g_1 + f_1\ast g_2 + f_2\ast g_1 + f_2\ast g_2,$$
and we easily see that $f_1\ast g_2, f_2\ast g_1, f_2\ast g_2 \in \mathscr{S}(\R^*,\CCC^{\infty}_c(M\times M,\hd)$. We now show that $f_1\ast g_1 \in \mathscr{S}(\mathbb{V},\hd)$. By construction, \(f_1\ast g_1\) has support included in \(\mathbb{V}\) and is Schwartz when \(t=0\) because \(f_{1|t=0}, g_{1|t=0}\) are Schwartz. We then show that \(f_1\ast g_1\) is smooth. We can approximate (in the sense of uniform convergence on every compact subset) both \(f_1\) and \(g_1\) by smooth compactly supported functions in \(\mathbb{V}\). The convolution product of these approximation is smooth with compact support because \(\CCC^{\infty}_c(\T_HM, \hd)\) is an algebra for the convolution, and the product converges to \(f_1\ast g_1\). In particular, \(f_1\ast g_1\) is smooth
Thus the only remaining property to verify is the support condition for $f_1\ast g_1$. Let $K,C \subset \mathcal{U}'$ be compact subsets and $T >0$ such that we have the inclusions $\supp(f_1) \subset \exp^{\nabla,\psi}(K_T), \supp(g_1) \subset \exp^{\nabla,\psi}(C_T)$. Define $K\widetilde{\oplus}C \subset \gt_HM$ as:
\begin{align*}
exp^{\psi\circ\nabla\circ\psi^{-1}}(\psi(K\widetilde{\oplus}C)) = &\left\lbrace \left(\exp^{\psi\circ\nabla\circ\psi^{-1}}_{\exp^{\psi\circ\nabla\circ\psi^{-1}}_x(\psi(\xi))}(\psi(\eta)),x\right)\right. ,\\ 
&\left. (x,\xi) \in K \ \& \ (\exp^{\psi\circ\nabla\circ\psi^{-1}}_x(\psi(\xi)),\eta) \in C  \right\rbrace .
\end{align*}
This set is constructed so that $\exp^{\nabla,\psi}((K\widetilde{\oplus}C)_T)$ would contain all possible products (in $\T_HM$) of elements in $\supp(f_1)$ and $\supp(g_1)$ hence it would contain $\supp(f_1\ast g_1)$. By construction the set $\exp^{\psi\circ\nabla\circ\psi^{-1}}(\psi(K\widetilde{\oplus}C))$ is contained in $\exp^{\nabla,\psi}(\mathbb{U})$ so that $K\widetilde{\oplus}C$ is well defined, compact and hence $f_1\ast g_1 \in \mathscr{S}(\mathbb{V})$.

Now the fact that $\CCC^{\infty}_c(\T_HM,\hd) \subset \mathscr{S}(\T_HM,\hd)$ is rather obvious. Moreover since we also have $\mathscr{S}(\T_HM,\hd) \subset L^1(\T_HM,\hd)$ then every continuous representation of $\T_HM$ extends continuously to $\mathscr{S}(\T_HM,\hd)$ and thus we get the inclusion $\mathscr{S}(\T_HM,\hd) \subset C^*(\T_HM)$.
\end{proof}

\begin{corollary}The subalgebra $\mathscr{S}(\R^*,\CCC^{\infty}_c(M\times M,\hd))$ is an ideal of the Schwartz algebra $\mathscr{S}(\T_HM,\hd)$.\end{corollary}

\section{The ideal J of Debord and Skandalis}

In this section we introduce the analog of the ideal $\mathcal{J}(G)$ of Debord and Skandalis (see \cite{debord2014adiab}) for $G = M\times M$ (although the idea directly generalizes to groupoids with filtered algebroid). Roughly speaking this ideal of $\mathscr{S}(\T_HM,\hd)$ corresponds, in the commutative case, to the functions that vanish at infinite order on the zero section of $TM$ seen as the zero fiber. In the noncommutative case we need to replace evaluation at zero by the trivial representation of the fiber. Elements of $\mathscr{S}(\R^*,\CCC^{\infty}_c(M\times M,\hd))$ already vanish with infinite order on $T_HM$ so they will satisfy this condition. Therefore we need to declare which functions of $\mathscr{S}(\mathbb{V},\hd)$ will be in the ideal for an exponential chart $\exp^{\nabla,\psi} \colon \mathbb{U} \isomto \mathbb{V}$. Using the exponential map we can define the functions on the ideal working on $\mathbb{U}$. This allows us to use the Fourier transform $\mathcal{F} \colon \mathscr{S}(\gt_HM\times \R,\hd) \isomto \mathscr{S}(\gt_HM^*\times \R)$. Now the trivial representation of a fiber corresponds to the evaluation at $0$ for the Fourier transform\footnote{Here we use the scalar Fourier transform. In the non-commutative setup a representation valued Fourier transform is often more suitable, see \cite{LipRos,FermanianFischer}. However, since we are only interested in the trivial representation (which corresponds to a single point through the orbit method) the scalar valued transform is more manageable.}. Consequently we can reformulate vanishing on the trivial representation by the vanishing of the Fourier transform at zero.

\begin{definition}Let $\exp^{\nabla,\psi}\colon \mathbb{U} \isomto \mathbb{V}$ be an exponential chart for $\T_HM$. A function $f \in \mathscr{S}(\T_HM,\hd)$ written as 
\[f = f_0 +f_1 \in \mathscr{S}(\mathbb{V},\hd) + \mathscr{S}(\R^*,\CCC^{\infty}_c(M\times M, \hd))\]
is in the subspace $\mathcal{J}_H(M)$ if $\F\left(f_0 \circ \exp^{\nabla,\psi}\circ\Psi\right)$ vanishes at all orders on the zero section of $\gt_HM^*\times \R$.
\end{definition}

\begin{proposition}\label{Normalizer}A function $f = (f_t)_{t\in \R} \in \mathscr{S}(\T_HM,\hd)$ is in $\mathcal{J}_H(M)$ if and only if for every $g \in \CCC^{\infty}_c(M\times M,\hd)$ the function $t\mapsto f_t\ast g$ extends to a function that belongs to $\mathscr{S}(\R^*,\CCC^{\infty}_c(M\times M,\hd))$.
\end{proposition}
\begin{proof}
Since $\mathscr{S}(\R^*,\CCC^{\infty}_c(M\times M,\hd))$ is an ideal, we can assume that the support of $f$ lies in an exponential chart 
$$\exp^{\nabla,\psi}\circ \Psi \colon \mathbb{U} \isomto \mathbb{V}.$$
Writing the Taylor series at 0 of $t \mapsto \F\left(f_t\circ\exp^{\nabla,\psi}\circ\Psi\right)(x,\tsp\diff\delta_t \eta)$, we see that every $t\mapsto f_t\ast g$ vanishes at order $k$ at $t = 0$ if and only if the Taylor coefficients up to order $k$ vanish as well.
\end{proof}

\begin{remmark}The proof actually shows that the functions $t\mapsto f_t \ast g$ have the same order of vanishing as the one of the Fourier transform of $f$ in local charts.\end{remmark}

\begin{corollary}The set $\mathcal{J}_H(M)$ is a *-ideal of $\mathscr{S}(\T_HM,\hd)$.\end{corollary}
\begin{proof}
This directly follows from the previous proposition and the fact that the subalgebra $\mathscr{S}(\R^*,\CCC^{\infty}_c(M\times M,\hd))$ is a *-ideal of $\mathscr{S}(\T_HM, \hd)$ itself.
\end{proof}

\section{Pseudodifferential operators as integrals}

We now adapt the main result of \cite{debord2014adiab} to the non-commutative case. To prove that a certain distribution on $M\times M$ is a pseudodifferential operator we will want to find a quasi-homogeneous extension to $\T_HM$. This is not very natural however and the natural extension is exactly homogeneous. We thus need a global version of Taylor's Lemma \ref{Taylor1}. Its proof is similar to the previous one that can be found in \cite[Prop. 2.4]{Taylor}, this global version can be found in \cite[Prop. 3.4]{AndroulidakisMohsenYunken}.

\begin{lemma}[Global Taylor's lemma]\label{Taylor2}Let $\exp^{\nabla,\psi}\colon \mathbb{U} \isomto \mathbb{V} \subset \T_HM$ be an exponential chart. Let $\PP \in \bbpsi^m_H(M)$ with support in $\mathbb{V}$. There is a (unique) smooth function $v \in \CCC^{\infty}((\gt_HM\times \R )\setminus (M\times \{(0,0)\}))$ such that:
\begin{itemize}
\item[i)]$v$ is homogeneous of degree $m$: $\forall \lambda > 0, v(x,\tsp\diff\delta_{\lambda}(\eta),t\lambda) = \lambda^m v(x,\eta,t)$.
\item[ii)]If $\chi \in \CCC^{\infty}_c(\gt_HM^*\times \R)$ equal to $1$ in the neighborhood of the zero section, then the map:
$$(x,\xi,t) \mapsto \exp^{\nabla,\psi*}\PP(x,\xi,t)- \int_{\eta\in \gt_{H,x}M^*}e^{\ii \eta(\xi)}(1-\chi)(x,\eta,t)v(x,\eta,t)$$
is a smooth section of $\Omega^{1/2}$ over $\gt_HM\times \R$. It has support included in $\gt_HM\times [-T;T]$ for some $T > 0$, and all the derivatives w.r.t. $x$ and $t$ are Schwartz in the $\xi$ direction uniformly in $x$ and $t$.
\end{itemize}
Conversely if a function $v$ satisfies $i)$ then one can find $\PP \in \bbpsi^m_H(M)$ such that $ii)$ is satisfied. 
\end{lemma}

\begin{theorem}\label{PseudodiffIntégrales}Let $m \in \CC$, $f\in \mathcal{J}_H(M)$. The convolution operator by $\int_0^{+\infty}t^{-m} f_t \frac{\diff t}{t}$ on $M\times M$ corresponds to a pseudodifferential operator of order $m$ in the filtered calculus. Its principal symbol is given as an element of $\mathcal{K}^m(T_HM)$ by $\int_0^{+\infty} t^{-m}\delta_{t*}f_0 \frac{\diff t}{t}$.
\end{theorem}

\begin{proof}
If $f\in \mathscr{S}(\R^*,\CCC^{\infty}_c(M\times M,\hd))$, then the integral converges absolutely. Proposition \ref{Normalizer} then gives meaning to the integral  $\int_0^{+\infty}t^{-m} f_t \frac{\diff t}{t}$ as a convolution operator. Moreover, we can assume that $f$ is supported in an exponential chart and work on $\gt_HM \times \R$. Using the converse of Taylor's Lemma, we need to find a homogeneous extension of the Fourier transform of the operator on $\mathbb{U}$, where $\mathbb{U}\isomto \mathbb{V}$ is the exponential neighborhood where $f$ is supported. We can rewrite the operator as
$$\int_0^{+\infty}s^{-m}\alpha_{s*}f_1\frac{\diff s}{s},$$
and thus consider the natural extension:
$$\PP_{t} := \int_0^{+\infty}s^{-m}\alpha_{s*}f_t \frac{\diff s}{s}.$$
This extension is exactly homogeneous of degree $m$ for the push-forward by the zoom action. Since $\alpha$ restricts to $\delta$ on the zero fiber we obtain that 
$$\PP_0 = \int_0^{+\infty} s^{-m}\delta_{s*}f_0 \frac{\diff s}{s}.$$
In order to use Taylor's Lemma, we consider $f$ as a function in $\mathscr{S}(\mathbb{U})$. We now need to show that the Fourier transform of $\PP$ defines a smooth function. We have $f_t \in \mathscr{S}(\gt_HM,\hd)$ so its Fourier transform is well defined. Recall that under the exponential map, the zoom action on $\gt_HM \times \R$ becomes $\beta$ with:
$$\beta_{\lambda}(x,\xi,t) = (x,\diff\delta_{\lambda}(\xi),\lambda^{-1}t).$$
The action on $\gt_HM^*\times \R$ is then $\tsp\beta$ where:
$$\tsp\beta_{\lambda}(x,\eta,t) = (x,\tsp\diff\delta_{\lambda}(\eta),\lambda t).$$
If $\mathcal{F} \colon \mathscr{S}(\gt_HM\times \R,\hd) \to \mathscr{S}(\gt_HM^*\times \R)$ denotes the Fourier transform in the $\gt_HM$ directions, then we have:
$$\mathcal{F}\circ\beta_{\lambda *} = \tsp\beta_{\lambda}^*\circ\mathcal{F}.$$
This result yields:
$$\widehat{\PP}(x,\eta,t) = \int_0^{+\infty}s^{-m}\hat{f}_{t s}(x,\tsp\diff\delta_{s}(\eta))\frac{\diff s}{s}.$$
Now $s \mapsto \hat{f}_{t s}(x,\tsp\diff\delta_{s}(\eta))$ is in $\mathscr{S}(\R^*_+)$ by assumption, thus the integral converges (we can factor out a high enough power of $s$ to compensate for the $s^{-(m+1)}$ and get the convergence at $0$). This convergence is uniform for $(x,\eta,t)$ in any compact subset of $(\gt_HM^*\times \R) \setminus (M\times \{0\})$. Therefore, the function $\widehat{\PP}$ is smooth on $(\gt_HM^*\times \R) \setminus (M\times \{0\})$ and homogeneous of degree \(m\) by construction. Therefore, the operator given by Taylor's Lemma gives a quasi homogeneous extension to $\PP_1$ (it is equal to $\PP_1$ at $t =1$ modulo a smoothing operator).
\end{proof}

\begin{remmark}We actually see that to obtain the convergence of the integral we only need $\hat{f}$ to vanish on the zero section at order superior or equal to $\Re(m) + 1$.\end{remmark}

\begin{lemma}\label{Lifting}Let $f_0 \in \mathscr{S}(T_HM,\hd)$ such that $\hat{f_0}$ vanishes on $M\times \{0\}$ at order $k \in \N \cup \{+\infty\}$, then there exists $f \in \mathscr{S}(\T_HM,\hd)$ such that, in an exponential chart, $(\hat{f_t})_{t\in \R}$ vanishes on $M\times \{0\} \subset \gt_HM \times \R$ at order $k$ and $f_0 = f(\cdot,0)$.
\end{lemma}
\begin{proof}
Let us consider a function $\vphi \in \CCC^{\infty}_c(\gt_HM)$ such that $\vphi(x,0) = 1$ and $\lim_{t\rightarrow+\infty}\widehat{\vphi}\circ \tsp\diff\delta_{t} = \delta_M$ where $\delta_M$ is the distribution corresponding to integrating along the zero section of $\gt_HM$ (it corresponds to the Dirac distribution at zero on each fiber of $\gt_HM$). Then we set:
$$\widetilde{f}(x,\xi,t)= f_0(x,\xi)\varphi(x,\diff\delta_t(\xi)).$$
We have $\widehat{\widetilde{f}_t}(x,\eta) = \widehat{f_0}(x,\eta)\ast_{ab}\varphi(x,\tsp\diff\delta_{t^{-1}}(\eta))$ where $\ast_{ab}$ is the convolution product on $\gt^*_HM$ seen as a bundle of abelian groups and we obtain the vanishing of $\widehat{\widetilde{f}}$ at the right order. However the function $\tilde{f}$ is not in the Schwartz class as $\tilde{f}(x,0,t) = f_0(x,0)$ is constant in $t$ hence not rapidly decreasing in general. This can be corrected by multiplying $\tilde{f}$ by a cutoff function $\chi \in \CCC^{\infty}_c(\R)$ that is constant equal to $1$ in a neighborhood of $0$.
\end{proof}

\begin{proposition}Every pseudodifferential operator can be written as in Theorem \ref{PseudodiffIntégrales}: if $P\in \Psi^m_H(M)$ there exists $f \in \mathcal{J}_H(M)$ such that 
$$P = \int_0^{+\infty}t^{-m}f_t\frac{\diff t}{t}.$$
\end{proposition}
\begin{proof}
The result is trivial if $P \in \CCC^{\infty}_c(M\times M,\hd)$. If $P$ is a pseudodifferential operator of order $m$ then let $\sigma^m(P) \in \Sigma^m(T_HM)$ be its principal symbol, its associated full symbol $v \in \CCC^{\infty}(\gt_HM^*\setminus (M\times \{0\})$ obtained by Lemma \ref{Taylor1} is homogeneous of order $m$ and can thus be written as $\int_0^{+\infty}t^{-m}\delta_t^*\hat{f_0}\frac{\diff t}{t}$ for a function $f_0 \in \mathscr{S}_0(\T_HM,\hd)$. We then use Lemma \ref{Lifting} to extend $f_0$ to a function $f \in \mathcal{J}_H(M)$. Then $P$ and $P_f :=  \int_0^{+\infty}t^{-m}f_t\frac{\diff t}{t}$ have the same principal symbol and thus belong to $\Psi^{m-1}_H(M)$. We then iterate the argument and approach $P$ with an asymptotic series (recall from \cite[Thm. 59]{van_erpyunken} that the pseudodifferential filtered calculus is asymptotically complete).
\end{proof}

\begin{lemma}Let $f,g \in \mathcal{J}_H(M)$ then $u \mapsto (f_t\ast g_{tu})_{t\in \R} \in \mathscr{S}(\R^*_+,\mathcal{J}_H(M))$.\end{lemma}
\begin{proof}
This is the same type of reasoning as in Proposition \ref{Normalizer} but applied to the groupoid $\T_HM$ instead of $M\times M$ (recall our construction is valid for an arbitrary groupoid). Here $\T_HM$ has its algebroid naturally filtered by the filtration on $TM$ and on $\gt_HM$, the $j$-th stratum of $\gt_HM$ being $\bigoplus_{i = 1}^j \faktor{H^i}{H^{i-1}}$.
\end{proof}

\begin{proposition}\label{SmoothMult}Let $f = (f_t)_{t\in \R} \in \mathcal{J}_H(M)$ and $P \in \Psi^m_H(M)$, then the family of functions $(f_t \ast P)_{t\in\R^*}$ extends to an element of $\mathcal{J}_H(M)$. Moreover its value at $t = 0$ is $f_0 \ast \sigma^m(P)$ with $\sigma^m(P) \in \mathcal{K}^m(T_HM)$ seen as a multiplier of $\mathscr{S}_0(T_HM)$.
\end{proposition}
\begin{proof}
Write $P$ as an integral $\int_0^{+\infty}t^{-m}g_t\frac{\diff t}{t}$ then use the previous Lemma with $f$ and $g$.
\end{proof}

\section{Completion and crossed product of the pseudodifferential algebra}

We now consider the $C^*$-completion of the previous algebras. From now on we restrict the tangent groupoid over $\R_+$ and denote it by 
$$\T_H^+M:= \T_HM_{|M\times \R_+}.$$
We also denote by $\mathcal{J}_H^+(M)$ the space of restriction of functions in $\mathcal{J}_H(M)$ to $\T_H^+M$. We can consider the respective completions of the algebras $\mathscr{S}(\T_H^+M,\hd)$, $\mathcal{J}^+_H(M)$, and $\mathscr{S}(\R^*_+,\CCC^{\infty}_c(M\times M,\hd)$. We obtain the $C^*$-algebras $C^*(\T_H^+M), C^*_0(\T_H^+M)$, and $\CCC_0(\R^*_+,\mathcal{K}(L^2(M)))$. They sit in similar exact sequences\footnote{The $(0)$ means that the sequences are exact with or without the $0$.} as in the commutative case:
$$\xymatrix{0 \ar[r] & \CCC_0(\R^*_+,\mathcal{K}(L^2(M))) \ar[r] & C^*_{(0)}(\T_H^+M) \ar[r] & C^*_{(0)}(T_HM) \ar[r] & 0.}$$
We also denote by $\Psi^*_H(M)$ the $C^*$-completion of $\Psi^0_H(M)$.

%\begin{lemma}[Factorisation Lemma]Every element $f \in \mathscr{S}(T_HM)$ can be written as a finite sum of the form $\sum_{i = 1}^k g_i^*\ast h_i$ with $g_1,\cdots,g_k,h_1,\cdots,h_k \in \mathscr{S}(T_HM)$. If $f \in \mathscr{S}_0(T_HM)$ we can moreover assume that each $h_i$ is in $\mathscr{S}_0(T_HM)$ and that $\widehat{g_i}$ vanishes on the zero section at any fixed finite order. \end{lemma}

%\begin{proposition}
%\begin{itemize}
%\item[a] If $g \in \mathcal{K}(L^2(M))$ and $f \in C^*_0(\T_H^+M)$ then $g\ast f \in \CCC_0(\R^*_+,\mathcal{K}(L^2(M))$. 
%\item[b] If $P \in \Psi^*_H(M)$ and $f \in C^*_0(\T_H^+M)$ then $f\ast P \in C^*_0(\T_H^+M)$.
%\end{itemize}
%\end{proposition}
%\begin{proposition}Let $f \in \mathcal{J}_H^+(M)$ and $s\in \CC$ with $\Re(s) \leq 0$. The integral $\int_0^{+\infty}t^{-s}f_t\frac{\diff t}{t}$ converges strictly (as an element of $\m(C^*(M\times M)) = \mathcal{B}(L^2(M))$).\end{proposition}

\begin{proposition}\label{ContMult} Let $P \in \Psi^0_H(M)$, the action of on $\mathcal{J}_H^+(M)$ defined in Proposition \ref{SmoothMult} for \(m=0\) extends to a multiplier of $C^*_0(\T_HM^+)$. Moreover, the resulting morphism $\Psi^0_H(M) \to \m(C^*_0(\T_HM^+))$ is continuous. If in addition $P \in \Psi^{-1}_H(M)$, then its action preserves the ideal $\CCC_0(\R^*_+,\mathcal{K}(L^2(M)))$.
\end{proposition}
\begin{proof}
Recall that $C^*_0(\T_H^+M)$ is a continuous field of $C^*$-algebras over $\R_+$ with fiber at $t\neq 0$ equal to $\mathcal{K}(L^2(M)) = C^*(M\times M)$ and its fiber at $t=0$ is $C^*_0(T_HM)$. We thus have for each $t\neq 0$, $\|(P\ast f)_t\| = \|P\|_{\Psi^*}\|f_t\|$ and $\|(P\ast f)_0\| \leq \|\sigma(P)\|\|f_0\| \leq \|P\|_{\Psi^*}\|f_0\|$. Here $\|\cdot\|_{\Psi^*}$ denotes the norm of $\Psi^*_H(M)$, i.e. the norm of the operator $P$ acting on $L^2(M)$ (this operator is bounded because \(P\) has order \(0\)). The last inequality follows from the continuity of the principal symbol map. From these inequalities it follows that:
$$\forall f \in \mathcal{J}_H^+(M), \|P\ast f\| \leq \|P\|_{\Psi^*}\|f\|.$$
Therefore the operator $(f_t)_{t\geq 0} \mapsto (P\ast f_t)_{t\geq 0}$ extends continuously to an operator on $C^*_0(\T_HM)$. The morphism $\Psi^0_H(M) \to \m(C^*_0(\T_HM))$ is continuous and extends to 
$$\Psi^*_H(M) \to \m(C^*_0(\T_HM)).$$
The last statement follows from the fact that the restriction of $P\ast f$ to the zero fiber is given by the action of the principal symbol (of order $0$) of $P$ on $f_0$.  This action is thus zero if $P \in \Psi^{-1}_H(M)$.
\end{proof}

\begin{remmark}More precisely, if $P \in \Psi^{-1}_H(M)$ and $f \in \mathcal{J}_H^+(M)$ then one can prove that $P\ast f \in \mathscr{S}(\R^*_+,\CCC^{\infty}_c(M\times M,\hd))$.\end{remmark}

Let $\Delta \in \Psi_H^q(M)$ be a positive Rockland differential operator (for some even $q>0$). Using Theorem \ref{ComplexOp} we have a continuous family of operators $\Delta^{\ii t/q} \in \Psi_H^{\ii t}(M)$. These operators act on $\mathcal{J}_H^+(M)$ as the order zero operators. Analogously to the previous proposition, this action extends to unitary multipliers of $C^*_0(\T_HM^+)$. The operators also give an $\R$-action on $\Psi^*_H(M)$ with $\Ad(\Delta^{\ii t/q})$. We thus get a *-homomorphism:
$$\Phi \colon \Psi^*_H(M) \rtimes \R \to C^*_0(\T_H^+M).$$

\begin{theorem}The morphism $\Phi$ preserves the respective exact sequences of $\Psi^*_H(M)$ and $C^*_0(\T_HM^+)$, i.e. there is a commutative diagram:
$$\xymatrix{0 \ar[r] & \CCC_0(\R^*_+,\mathcal{K}(L^2(M))) \ar@{=}[d] \ar[r] & \Psi^*_H(M) \rtimes \R \ar[d]^{\Phi} \ar[r] & \Sigma(T_HM) \rtimes \R \ar[d]^{\Phi_0} \ar[r] & 0 \\
0 \ar[r] & \CCC_0(\R^*_+,\mathcal{K}(L^2(M))) \ar[r] & C^*_{0}(\T_H^+M) \ar[r] & C^*_{0}(T_HM) \ar[r] & 0.}$$
Here $\Phi_0 \colon \Sigma(G) \rtimes \R \to C^*_0(T_HM)$ is the isomorphism of the Section \ref{Section:IsomorphismSymbols}. In particular $\Phi$ is an isomorphism.
\end{theorem}
\begin{proof}
This is a direct consequence of the previous proposition. The fact that the crossed product is trivial on the ideal in the first row is because each $\Delta^{\ii t/q}$ extends to a multiplier of $\mathcal{K}(L^2(M)) = \overline{\Psi^{-1}_H(M)}$. Using Proposition \ref{SmoothMult}, we get that $\Phi_0$ is the isomorphism of the Section \ref{Section:IsomorphismSymbols} using the symbol \(\Delta_0 := \sigma^q(\Delta)\). The fact that $\Phi$ is an isomorphism then follows from the short five lemma.
\end{proof}

From this result we can state another one similar to Debord and Skandalis original approach, also obtained in the filtered case by Ewert in \cite{Ewert}:

\begin{corollary}\label{Cor:MoritaEq}There is a canonical bimodule yielding a Morita equivalence between the algebras $\Psi^*_H(M)$ and $C^*_0(\T_HM)\rtimes_{\alpha} \R^*_+$ preserving their respective exact sequences.\end{corollary}
\begin{proof}
The zoom action $\alpha$ is dual to the one defined with $(\Delta^{\ii t/q})_{t\in \R}$ thus this is just an application of Takai duality.
\end{proof}

\begin{remmark}
    One can also describe the equivalence bimodule directly as in \cite{debord2014adiab}. Using Proposition \ref{SmoothMult}, we get a (right) \(\Psi^0_H(M)\)-module structure on \(\mathcal{J}^+_H(M)\). Using the \(\Psi^0_H(M)\) values product:
    \[(f,g) \mapsto \int_0^{+\infty} f_t^*\ast g_t \frac{\diff t}{t},\]
    we can complete \(\mathcal{J}_H^+(M)\) into a full \(\Psi^*_H(M)\)-Hilbert module, \(\mathcal{E}_H\). The approach of Debord and Skandalis would then be to show that the natural map \(C^*_0(\T_H^+M)\rtimes \R^*_+ \to \mathcal{L}(\mathcal{E}_H)\) induces an isomorphism:
    \[C^*_0(\T_H^+M)\rtimes \R^*_+ \cong \mathcal{K}(\mathcal{E}_H).\]
    Then one can show that the module \(\mathcal{E}_H\) is stable, in particular, \(\mathcal{K}(\mathcal{E}_H) \cong \mathcal{K} \otimes \Psi^*_H(M)\) and we get the desired Morita equivalence.
\end{remmark}

The previous result allows for some index theoretic consequences\footnote{The first KK-equivalence is also in \cite{Ewert}.}:

\begin{proposition}Let $(M,H)$ be a compact filtered manifold then the algebras of symbols $\Sigma(T_HM)$ and $\Sigma(TM)$ are KK-equivalent. The pseudodifferential algebras $\Psi^*_H(M)$ and $\Psi^*(M)$ are KK-equivalent as well.
\end{proposition}
\begin{proof}
We use the last corollary to get a Morita equivalence 
$$\Sigma(T_HM) \sim_{mor} C^*_0(T_HM)\rtimes \R^*_+.$$
The same goes in the commutative case with $\Sigma(TM) \sim_{mor} C^*_0(TM)\rtimes  \R^*_+$ (which is just the isomorphism $\CCC_0(\mathbb{S}^*M) \cong \CCC_0(T^*M\setminus 0)\rtimes \R^*_+$). Now we can use the Connes-Thom isomorphism for both groups bundles $T_HM$ and $TM$ to get a KK-equivalence $C^*(TM)\sim_{KK}C^*(T_HM)$ which restricts to the kernels of trivial representations in a $\R^*_+$-equivariant way, we thus get $C^*_0(TM) \sim_{KK^{\R^*_+}} C^*_0(T_HM)$ and thus the desired KK-equivalence by composition (remember a Morita equivalence induces a KK-equivalence).

For the pseudodifferential algebras we will use the Morita equivalence of Corollary \ref{Cor:MoritaEq}: $\Psi^*_H(M) \sim_{mor} C^*_0(\T^+_HM)$. We now need to show a KK-equivalence: 
$$C^*_0(\T^+_HM) \sim_{KK} C^*_0(\T^+M).$$
Using the (classical) adiabatic groupoid of $\T_HM$ \footnote{One could also use the filtered adiabatic groupoid of $\T M$ for the filtration given by the $H^i \times \R$. This results in an isomorphic groupoid under the flip of coordinates $(s,t)\mapsto (t,s)$ on $\R^2$.} and restricting it over $[0;1]^2$, we obtain a groupoid $\mathbb{G}\rr M\times [0;1]^2$. It has the following restrictions with coordinates $(s,t) \in [0;1]^2$:
\begin{itemize}
\item the restriction to a fiber for $s \neq 0$ of $\mathbb{G}$ is $\T_HM_{|[0;1]}$
\item the restriction to the fiber $s = 0$ of $\mathbb{G}$ is $TM\times [0;1]$ (after a choice of splitting for $\gt_HM$).
\item the restriction to a fiber for $t \neq 0$ of $\mathbb{G}$ is $\T M_{|[0;1]}$
\item the restriction to the fiber $t = 0$ of $\mathbb{G}$ is $T_HM^{ad}_{|[0;1]}$
\end{itemize}
Now using the $KK$-equivalences obtained by the exact sequences induces by the row and columns of the square we get the KK-equivalences:
\begin{align*}
C^*(\T_HM_{|[0;1]}) &\sim C^*(TM\times [0;1]) \\
				   &\sim C^*(TM) \\
				   &\sim C^*(T_HM^{ad}_{|[0;1]}) \\
				   &\sim C^*(\T M_{|[0;1]})
\end{align*}
This KK-equivalence commutes with the evaluation maps on the same fiber and thus restricts to a KK-equivalence between the groupoid algebras $C^*(\T_HM_{|[0;1)})$ and $C^*(\T M_{|[0;1)})$. We can then use a homeomorphism between $[0;1)$ and $\R_+$ to get a KK-equivalence $C^*(\T^+_HM) \sim C^*(\T^+M)$. By construction this equivalence\footnote{The same exact sequences used to obtain the KK-equivalences can be written for the $C^*_0$ algebras.} restricts to a KK-equivalence between $C^*_0(\T^+_HM)$ and $C^*_0(\T^+M)$  in an $\R^*_+$ equivariant way. We thus obtain the KK-equivalence between their crossed products $C^*_0(\T^+_HM)\rtimes \R^*_+$ and $C^*_0(\T^+M)\rtimes \R^*_+$ and thus between $\Psi^*_H(M)$ and $\Psi^*(M)$. \qedhere
\end{proof}

Roughly speaking, the last proposition shows that one could not hope to obtain further index theoretic invariants from the filtered calculus, other than the ones that could previously be obtained from the classical pseudodifferential calculus. This however does not help when one wants to compute the index of an actual operator, as the difficulty then lies in computing the inverse of the Connes-Thom isomorphism that was used here (see \cite{BaumVanerp,mohsen2020index}).
\nocite{*}
\bibliographystyle{plain}
\bibliography{Ref}

\end{document}